\documentclass[11pt]{amsart}

\usepackage{amsfonts}
\usepackage{hyperref}
\usepackage{amsmath,amssymb}
\usepackage{enumerate,graphicx}
\usepackage[square,numbers]{natbib}
\usepackage{pst-node,pst-tree}
\usepackage{auto-pst-pdf}

\newtheorem{theorem}{Theorem}[section]

\newtheorem{proposition}[theorem]{Proposition}
\newtheorem{corollary}[theorem]{Corollary}
\newtheorem{lemma}[theorem]{Lemma}

\theoremstyle{definition}

\newtheorem{definition}[theorem]{Definition}
\newtheorem{example}[theorem]{Example}

\newtheorem{problem}[theorem]{Problem}
\newtheorem{remark}[theorem]{Remark}

\begin{document}

\title[MCNN Structure: Complexity Between Two Layers]{On the structure of multi-layer cellular neural networks: Complexity between two layers} %



\author{Jung-Chao Ban}
\address[Jung-Chao Ban]{Department of Applied Mathematics, National Dong Hwa University, Hualien 970003, Taiwan, ROC.}
\email{jcban@mail.ndhu.edu.tw}

\author{Chih-Hung Chang}
\address[Chih-Hung Chang]{Department of Applied Mathematics, National University of Kaohsiung, Kaohsiung 81148, Taiwan, ROC.}
\email{chchang@nuk.edu.tw}

\maketitle

\begin{abstract}
Let $\mathbf{Y}$ be the solution space of an $n$-layer cellular neural network, and let $\mathbf{Y}^{(i)}$ and $\mathbf{Y}^{(j)}$ be the hidden spaces, where $1 \leq i, j \leq n$. ($\mathbf{Y}^{(n)}$ is called the output space.) The classification and the existence of factor maps between two hidden spaces, that reaches the same topological entropies, are investigated in [Ban et al., J.~Differential Equations \textbf{252}, 4563-4597, 2012]. This paper elucidates the existence of factor maps between those hidden spaces carrying distinct topological entropies. For either case, the Hausdorff dimension $\dim \mathbf{Y}^{(i)}$ and $\dim \mathbf{Y}^{(j)}$ can be calculated. Furthermore, the dimension of $\mathbf{Y}^{(i)}$ and $\mathbf{Y}^{(j)}$ are related upon the factor map between them.
\end{abstract}

\section{Introduction}

Multi-layer cellular neural networks (MCNNs) are large aggregates of analogue circuits presenting themselves as arrays of identical cells which are locally coupled. MCNNs have been widely applied in studying the signal propagation between neurons, and in image processing, pattern recognition and information technology \cite{ABFM-ITCSIFTA1998,CR-2002,CY-ITCS1988a,CC-ITCS1995,CRC-ITCS1993,Li-NC2009,Mur-IJCM2010,PZL-FI2009,XYSV-ITCSIRP2004,YNU-ITF2002}. A One-dimensional MCNN is realized as
\begin{equation}\label{eq-general-system}
\frac{d x^{(\ell)}_i}{dt} = - x^{(\ell)}_i  + \sum_{|k| \leq d} a^{(\ell)}_k y^{(\ell)}_{i+k} + \sum_{|k| \leq d} b^{(\ell)}_k u^{(\ell)}_{i+k} + z^{(\ell)},
\end{equation}
for some $d \in \mathbb{N}, 1 \leq \ell \leq n \in \mathbb{N}, i \in \mathbb{Z}$, where
\begin{equation}
u^{(\ell)}_i = y^{(\ell-1)}_i \text{ for } 2 \leq \ell \leq n, \quad u^{(1)}_i = u_i, \quad x_i^{(\ell)}(0) = x_{i,0}^{(\ell)},
\end{equation}
and
\begin{equation}
y=f(x) = \frac{1}{2} (|x+1|-|x-1|)
\end{equation}
is the output function.

The stationary solutions $\bar{x} = (\bar{x}^{(\ell)}_i)$ of (\ref{eq-general-system}) are essential for understanding the system, and their outputs $\bar{y}^{(\ell)}_i = f(\bar{x}^{(\ell)}_i)$ are called \emph{output patterns}. A \emph{mosaic solution} $(\bar{x}^{(\ell)}_i)$ is a stationary solution satisfying $|\bar{x}^{(\ell)}_i| > 1$ for all $i, \ell$ and the output of a mosaic solution is called a mosaic output pattern. Mosaic solutions are crucial for studying the complexity of \eqref{eq-general-system} due to their asymptotical stability \cite{CM-ITCSIFTaA1995,CMV-IJBCASE1996,CMV-RCD1996,CS-SJAM1995,CCTS-IJBCASE1996,GBC-ITCSIRP2004,IC-IJBCASE2004,JL-SJAM2000,KM-IJBCASE2010,LH-AMAS2003,LS-IJBCASE1999,Shi-SJAM2000}. In a MCNN system, the ``status" of each cell is taken as an input for a cell in the next layer except for those cells in the last layer. The results that can be recorded are the output of the cells in the last layer. Since the phenomena that can be observed are only the output patterns of the $n$th layer, the $n$th layer of \eqref{eq-general-system} is called the \emph{output layer}, while the other $n-1$ layers are called \emph{hidden layers}.

We remark that, except from mosaic solutions exhibiting key features of MCNNs, mosaic solutions themselves are constrained by the so-called ``separation property'' (cf.~\cite{BC-IJBCASE2009, BC-NPL2014}). This makes the investigation more difficult. Furthermore, the output patterns of mosaic solutions of a MCNN can be treated as a cellular automaton. For the discussion of systems satisfying constraints and cellular automata, readers are referred to Wolfram's celebrated book \cite{Wol-2002}. (The discussion of constrained systems is referred to chapter 5.)

Suppose $\mathbf{Y}$ is the solution space of a MCNN. For $\ell = 1, 2, \ldots, n$, let
$$
\mathbf{Y}^{(\ell)} = \{\cdots y_{-1}^{(\ell)} y_0^{(\ell)} y_1^{(\ell)} \cdots\}
$$
be the space which consists of patterns in the $\ell$th layer of $\mathbf{Y}$, and let $\phi^{(\ell)}: \mathbf{Y} \to \mathbf{Y}^{(\ell)}$ be the projection map. Then $\mathbf{Y}^{(n)}$ is called the \emph{output space} and $\mathbf{Y}^{(\ell)}$ is called the ($\ell$th) \emph{hidden space} for $\ell = 1, 2, \ldots, n-1$. It is natural to ask whether there exists a relation between $\mathbf{Y}^{(i)}$ and $\mathbf{Y}^{(j)}$ for $1 \leq i \neq j \leq n$. Take $n = 2$ for instance; the existence of a map connecting $\mathbf{Y}^{(1)}$ and $\mathbf{Y}^{(2)}$ that commutes with $\phi^{(1)}$ and $\phi^{(2)}$ means the \emph{decoupling} of the solution space $\mathbf{Y}$. More precisely, if there exists $\pi_{12}: \mathbf{Y}^{(1)} \to \mathbf{Y}^{(2)}$ such that $\pi_{12} \circ \phi^{(1)} = \phi^{(2)}$, then $\pi_{12}$ enables the investigation of structures between the output space and hidden space. A serial work is contributed for this purpose.

At the very beginning, Ban \emph{et al.}~\cite{BCLL-JDE2009} demonstrated that the output space $\mathbf{Y}^{(n)}$ is topologically conjugated to a one-dimensional sofic shift. This result is differentiated from earlier research which indicated that the output space of a $1$-layer CNN without input is topologically conjugated to a Markov shift (also known as a shift of finite type). Some unsolved open problems, either on the mathematical or on the engineering side, have drawn interest since then. An analogous argument asserts that every hidden space $\mathbf{Y}^{(\ell)}$ is also topologically conjugated to a sofic shift for $1 \leq \ell \leq n-1$, and the solution space $\mathbf{Y}$ is topologically conjugated to a subshift of finite type. More than that, the topological entropy and dynamical zeta function of $\mathbf{Y}^{(\ell)}$ and $\mathbf{Y}$ are capable of calculation. A novel phenomenon, the asymmetry of topological entropy. It is known that a nonempty insertive and extractive language $\mathcal{L}$ is regular if and only if $\mathcal{L}$ is the language of a sofic shift; namely, $\mathcal{L} \subseteq \bigcup_{i \geq 0} B_i(X)$ for some sofic shift $X$, where
$$
B_i(X) = \{x_1 x_2 \cdots x_i: (x_k)_{k \in \mathbb{Z}} \in X\}.
$$
Therefore, elucidating sofic shifts is equivalent to the investigation of regular languages. Readers are referred to \cite{BP-2011} and the references therein for more details about the illustration of languages and sofic shifts.

Followed by \cite{BCL-JDE2012}, the classification of the hidden and output spaces is revealed for those spaces reaching the same topological entropy. Notably, the study of the existence of $\pi_{ij}: \mathbf{Y}^{(i)} \to \mathbf{Y}^{(j)}$ for some $i,j$ is equivalent to illustrating whether there is a map connecting two sofic shifts. Mostly it is difficult to demonstrate the existence of such maps. The authors have provided a systematic strategy for determining whether there exists a map between $\mathbf{Y}^{(i)}$ and $\mathbf{Y}^{(j)}$. More than that, the explicit expression of $\pi_{ij}$ is unveiled whenever there is a factor-like matrix $E$ (defined later).

The present paper, as a continuation of \cite{BCL-JDE2012, BCLL-JDE2009}, is devoted to investigating the Hausdorff dimension of the output and hidden spaces. We emphasize that, in this elucidation, those spaces need not attain the same topological entropy. In addition to examining the existence of maps between $\mathbf{Y}^{(i)}$ and $\mathbf{Y}^{(j)}$ (for the case where the topological entropies of two spaces are distinct), the complexity of the geometrical structure is discussed. The Hausdorff dimension of a specified space is an icon that unveils the geometrical structure and helps with the description of the complexity. This aim is the target of this study.

Furthermore, aside from the existence of factor maps between $\mathbf{Y}^{(i)}$ and $\mathbf{Y}^{(j)}$, the correspondence of the Hausdorff dimension is of interest to this study. Suppose there exists a factor map $\pi_{ij}: \mathbf{Y}^{(i)} \to \mathbf{Y}^{(j)}$, the Hausdorff dimension of $\mathbf{Y}^{(i)}$ and $\mathbf{Y}^{(j)}$ are related under some additional conditions (see Theorems \ref{main-thm-FSE} and \ref{main-thm-ITO}). More explicitly, it is now known that in many examples the calculation of the Hausdorff dimension of a set is closely related to the maximal measures (defined later) of its corresponding symbolic dynamical system (cf.~\cite[Theorem 13.1]{P-1997} for instance). Theorems \ref{main-thm-FSE} and \ref{main-thm-ITO} also indicate that the Hausdorff dimension of $\mathbf{Y}^{(j)}$ is the quotient of the measure-theoretic entropy $h_{\pi_{ij} \nu^{(i)}}(\mathbf{Y}^{(j)})$ and the metric of $\mathbf{Y}^{(j)}$, where $\nu^{(i)}$ is the maximal measure of $\mathbf{Y}^{(i)}$. Notably, such a result relies on whether the push-forward measure $\pi_{ij} \nu^{(i)}$ of $\nu^{(i)}$ under the factor map $\pi_{ij}$ remains a maximal measure. We propose a methodology so that all the conditions are checkable, and the Hausdorff dimension $\dim \mathbf{Y}^{(\ell)}$ can be formulated accurately for $1 \leq \ell \leq n$.

\begin{figure}
\begin{center}
\includegraphics[scale=0.7]{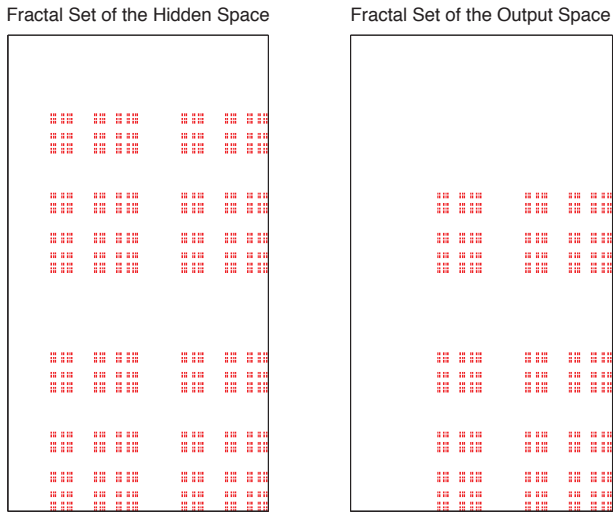}
\end{center}
\caption{The fractal sets of the hidden and output spaces of a MCNN with templates $[a^{(1)}, a_r^{(1)}, z^{(1)}] = [2.9, 1.7, 0.1]$ and $[a^{(2)}, a_r^{(2)}, b^{(2)}, b_r^{(2)}, z^{(2)}] = [-0.3, -1.2, 0.7, 2.3, 0.9]$. Each fractal set is a subspace of  $[0, 1] \times [0, 1]$, and is derived from the expansion $\Phi^{(i)}(x) = (\Sigma_{k \geq 0} \frac{x_k}{m_i^{k+1}}, \Sigma_{k \leq 0} \frac{x_k}{m_i^{|k|+1}})$ for $x \in \mathbf{Y}^{(i)}$, $m_i = |\mathcal{A}(\mathbf{Y}^{(i)})|$, and $i = 1, 2$. It is seen that the expansion map $\Phi^{(i)}: \mathbf{Y}^{(i)} \to [0, 1] \times [0, 1]$ is one-to-one almost everywhere, and hence does not make an impact on the discussion of the Hausdorff dimension. The figures come from repeating $9$ operations based on the basic set of admissible local patterns. See Example \ref{eg-Y1Y2-SFT} for more discussion.}
\label{fig-551429}
\end{figure}

Figure \ref{fig-551429} illustrates the fractal sets of the hidden and output spaces (namely, $\mathbf{Y}^{(1)}$ and $\mathbf{Y}^{(2)}$) of a two-layer CNN. It is seen that $\mathbf{Y}^{(1)}$ and $\mathbf{Y}^{(2)}$ are entirely two different spaces. Aside from calculating the Hausdorff dimension of these spaces, it is interesting to investigate whether there is a map connecting $\mathbf{Y}^{(1)}$ and $\mathbf{Y}^{(2)}$, and how $\dim \mathbf{Y}^{(1)}$ is related to $\dim \mathbf{Y}^{(2)}$. See Example \ref{eg-Y1Y2-SFT} for more details.

In the mean time, we want to mention some further issues that are related to our elucidation and which have caused widespread attention recently. One of them is the investigation of the so-called \emph{sofic measure} or \emph{hidden Markov measure}. Let $\mu$ be a Markov measure on $\mathbf{Y}$. The push-forward measure $\phi^{(\ell)} \mu$, defined by $(\phi^{(\ell)} \mu) (\mathcal{O}) = \mu ((\phi^{(\ell)})^{-1} \mathcal{O})$ for all Borel set $\mathcal{O}$ in $\mathbf{Y}^{(\ell)}$, is called a sofic measure or a hidden Markov measure. There have been piles of papers about sofic measures written in the past decades. A concerned question is under what condition the push-forward measure of a Markov measure remains a Markov measure. To be more specific, we are interested in which properties a sofic measure would satisfy. This elucidation focuses on the study of the measures on the hidden/output space. Recalling that the hidden/output space is a factor of the solution space, it follows that the investigation of the  measures on the hidden/output space is equivalent to the investigation of sofic measures. We propose a methodology to verify when a sofic measure is reduced to be Markov. In this case, the explicit form of a maximal measure and the Hausdorff dimension of the hidden/output space are formulated. For more discussion of the hidden Markov measures, the reader is referred to \cite{BCC-2012,BP-2011,KT-MAMS1985} and the references therein.

It is known that the tiling problem is undecidable. As an application, it is of interest to investigate the decidability of the language of a sofic shift which can be realized as a hidden or output space of a MCNN. The related work is in preparation.

The rest of this investigation is organized as follows. A brief recall of \cite{BCL-JDE2012, BCLL-JDE2009} and some definitions and notations are given in Section 2. The main theorems (Theorems \ref{main-thm-FSE} and \ref{main-thm-ITO}) for $2$-layer CNNs are also stated therein. Section 3 analyzes the existence of factor maps that connect two spaces and the hidden Markov measures. The proofs of the main theorems are illustrated there. Some examples are given in Section 4. We generalize Theorems \ref{main-thm-FSE} and \ref{main-thm-ITO} to general MCNNs in Section 5. Figure \ref{fig-flow-chart} provides the flow chart of the present investigation. Section 6 is saved for the conclusion and further problems.

\section{Main Results and Preliminaries}

Due to this paper being a continuation of \cite{BCL-JDE2012}, the upcoming section intends to give a brief review of \cite{BCL-JDE2012} and illustrates the main results of our study. For the self-containment of the present investigation, we recall some definitions and known results for symbolic dynamical systems and MCNNs. The reader is referred to \cite{BCL-JDE2012, BCLL-JDE2009, LM-1995} and the references therein for more details.

\subsection{Multi-layer Cellular Neural Networks}

Since an elucidation of two-layer CNNs is essential for the study of MCNNs, we refer MCNNs to two-layer CNNs and focus on them in the rest of this paper unless otherwise stated. A two-layer cellular neural network is realized as
\begin{equation}\label{eq-2layer-mcnn}
\left\{
\begin{split}
\frac{d x_i^{(1)}}{dt} &= - x_i^{(1)} + \sum_{|k| \leq d} a^{(1)}_k y_{i+k}^{(1)} + \sum_{|\ell| \leq d} b^{(1)}_{\ell} u_{i+\ell}^{(1)}  + z^{(1)}, \\
\frac{d x_i^{(2)}}{dt} &= - x_i^{(2)} + \sum_{|k| \leq d} a^{(2)}_k y_{i+k}^{(2)} + \sum_{|\ell| \leq d} b^{(2)}_{\ell} u_{i+\ell}^{(2)} + z^{(2)},
\end{split}
\right.
\end{equation}
for some $d \in \mathbb{N}$, and $u_{i}^{(2)} = y_{i}^{(1)}$ for $i \in \mathbb{Z}$; $\mathbb{N}$ denotes the positive integers and $\mathbb{Z}$ denotes the integers. The prototype of \eqref{eq-2layer-mcnn} is
$$
\frac{d x_i}{dt} = -x_i + \sum_{|k| \leq d} a_k y_{i+k} + \sum_{|\ell| \leq d} b_{\ell} u_{i+\ell} + z.
$$
Here $A = [-a_d, \cdots, a_d], B = [-b_d, \cdots, b_d]$ are the \emph{feedback} and \emph{controlling templates}, respectively. $z$ is the \emph{threshold}, and $y_i = f(x_i) = \frac{1}{2} (|x_i+1| - |x_i-1|)$ is the output of $x_i$. The quantity $x_i$ represents the state of the cell at $i$ for $i \in \mathbb{Z}$. The output of a stationary solution $\bar{x} = (\bar{x}_i)_{i \in \mathbb{Z}}$ is called a output pattern. A \emph{mosaic solution} $\bar{x}$ satisfies $|\bar{x}_i| > 1$ and its corresponding pattern $\bar{y}$ is called a \emph{mosaic output pattern}. Consider the mosaic solution $\bar{x}$, the necessary and sufficient condition for state ``$+$" at cell $C_i$, i.e., $\bar{y}_i = 1$, is
\begin{equation}\label{eq-cnn-state+}
a - 1 + z > -(\sum_{0 < |k| \leq d} a_k \bar{y}_{i+k} + \sum_{|\ell| \leq d} b_{\ell} u_{i+\ell}),
\end{equation}
where $a = a_0$. Similarly, the necessary and sufficient conditions for state ``$-$" at cell $C_i$, i.e., $\bar{y}_i = -1$, is
\begin{equation}\label{eq-cnn-state-}
a - 1 - z > \sum_{0 < |k| \leq d} a_k \bar{y}_{i+k} + \sum_{|\ell| \leq d} b_{\ell} u_{i+\ell}.
\end{equation}

For simplicity, denoting $\bar{y}_i$ by $y_i$ and rewriting the output patterns $y_{-d} \cdots y_0 \cdots y_d$ coupled with input $u_{-d} \cdots u_0 \cdots u_d$ as
\begin{equation}
\boxed{y_{-d} \cdots y_{-1} y_0 y_1 \cdots y_d \atop \displaystyle u_{-d} \cdots u_{-1} u_0 u_1 \cdots u_d} \equiv y_{-d} \cdots y_d \diamond u_{-d} \cdots u_d \in \{-1, 1\}^{\mathbb{Z}_{(2d+1)\times 2}}.
\end{equation}
Let
$$
V^n = \{ v \in \mathbb{R}^n : v = (v_1, v_2, \cdots, v_n), \text{ and } |v_i| = 1, 1 \leq i \leq n \},
$$
where $n = 4d+1$, \eqref{eq-cnn-state+} and \eqref{eq-cnn-state-} can be rewritten in a compact form by introducing the following notation.

Denote $\alpha = (a_{-d}, \cdots, a_{-1}, a_1, \cdots, a_d)$, $\beta = (b_{-d}, \cdots, b_d)$. Then, $\alpha$ can be used to represent $A'$, the surrounding template of $A$ without center, and $\beta$ can be used to represent the template $B$. The basic set of admissible local patterns with ``$+$" state in the center is defined as
\begin{equation}
\mathcal{B}(+, A, B, z) = \{v \diamond w \in V^n : a - 1 + z > -(\alpha \cdot v + \beta \cdot w) \},
\end{equation}
where ``$\cdot$" is the inner product in Euclidean space. Similarly, the basic set of admissible local patterns with ``$-$" state in the center is defined as
\begin{equation}
\mathcal{B}(-, A, B, z) = \{v' \diamond w' \in V^n : a - 1 - z > \alpha \cdot v + \beta \cdot w \}.
\end{equation}
Furthermore, the admissible local patterns induced by $(A, B, z)$ can be denoted by
\begin{equation}
\mathcal{B}(A, B, z) = (\mathcal{B}(+, A, B, z), \mathcal{B}(-, A, B, z)).
\end{equation}
It is shown that the parameter space can be partitioned into finite equivalent subregions, that is, two sets of parameters induce identical basic sets of admissible local patterns if they belong to the same partition in the parameter space. Moreover, the parameter space of a MCNN is also partitioned into finite equivalent subregions \cite{BCLL-JDE2009}.

Suppose a partition of the parameter space is determined, that is, the templates
$$
A^{(\ell)} = [a^{(\ell)}_{-d}, \cdots, a^{(\ell)}_{d}], \quad B^{(\ell)} = [b^{(\ell)}_{-d}, \cdots, b^{(\ell)}_{d}], \quad z^{(\ell)} \qquad \ell = 1, 2
$$
are given. A stationary solution
$
\mathbf{x} = \mathbf{x}^{(2)} \diamond \mathbf{x}^{(1)}
= \begin{pmatrix}
    x^{(2)}_i \\
    x^{(1)}_i \\
  \end{pmatrix}_{i \in \mathbb{Z}}
$
is called mosaic if $|x^{(\ell)}_i| > 1$ for $\ell = 1, 2$ and $i \in \mathbb{Z}$. The output
$
\mathbf{y} = \mathbf{y}^{(2)} \diamond \mathbf{y}^{(1)}
= \begin{pmatrix}
    y^{(2)}_i \\
    y^{(1)}_i \\
  \end{pmatrix}_{i \in \mathbb{Z}}
$
of a mosaic solution $\mathbf{x}$ is called a mosaic pattern.

Suppose $\mathcal{B}$ is the basic set of admissible local patterns of a MCNN. Since \eqref{eq-2layer-mcnn} is spatial homogeneous, that is, the templates of \eqref{eq-2layer-mcnn} are fixed for each cell, the \emph{solution space} $\mathbf{Y} \subseteq \{-1, 1\}^{\mathbb{Z}_{\infty \times 2}}$ is determined by $\mathcal{B}$ as
$$
\mathbf{Y}= \left\{\mathbf{y}^{(2)} \diamond \mathbf{y}^{(1)}: y^{(2)}_{i-d} \cdots y^{(2)}_i \cdots y^{(2)}_{i+d} \diamond
y^{(1)}_{i-d} \cdots y^{(1)}_i \cdots y^{(1)}_{i+d} \in \mathcal{B} \text{ for } i \in \mathbb{Z}\right\}.
$$
Moreover, the \emph{output space} $\mathbf{Y}^{(2)}$ and the \emph{hidden space} $\mathbf{Y}^{(1)}$ are defined by
$$
\mathbf{Y}^{(2)} = \left\{\mathbf{y} \in \{-1, 1\}^{\mathbb{Z}} : \mathbf{y} \diamond \mathbf{u} \in \mathbf{Y} \text{ for some } \mathbf{u}\right\}
$$
and
$$
\mathbf{Y}^{(1)} = \left\{\mathbf{u} \in \{-1, 1\}^{\mathbb{Z}} : \mathbf{y} \diamond \mathbf{u} \in \mathbf{Y} \text{ for some } \mathbf{y}\right\}
$$
respectively. In \cite{BCL-JDE2012, BCLL-JDE2009}, the authors demonstrated that $\mathbf{Y}$ is a \emph{shift of finite type} (SFT) and $\mathbf{Y}^{(1)}, \mathbf{Y}^{(2)}$ are both sofic shifts. In general, for $i = 1, 2$, the factor $\phi^{(i)}: \mathbf{Y} \to \mathbf{Y}^{(i)}$ is not even a finite-to-one surjective map. Furthermore, for $i = 1, 2$, there is a \emph{covering space} $W^{(i)}$ of $\mathbf{Y}^{(i)}$ and a finite-to-one factor $\phi^{(i)}: W^{(i)} \to \mathbf{Y}^{(i)}$ with $W^{(i)}$ being a SFT. (We abuse the finite-to-one factor $\phi^{(i)}$ rather than $\widehat{\phi}^{(i)}$ to ease the use of notation.) For a topological space $Y$, we say that $X$ is a covering space of $Y$ if there exists a continuous onto map $\phi: X \to Y$ which is locally homeomorphic.

A quantity that describes the complexity of a system is \emph{topological entropy}. Suppose $X$ is a shift space. Denote $\Gamma_k(X)$ the cardinality of the collection of words of length $k$. The topological entropy of $X$ is then defined by
$$
h(X) = \lim_{k \to \infty} \frac{\log \Gamma_k(X)}{k}.
$$
Whenever the hidden space $\mathbf{Y}^{(1)}$ and the output space $\mathbf{Y}^{(2)}$ reach the same topological entropy, $\mathbf{Y}^{(1)}$ and $\mathbf{Y}^{(2)}$ are \emph{finite shift equivalent} (FSE) \cite{BCL-JDE2012}. Herein two spaces $X$ and $Y$ are FSE if there is a triple $(Z, \phi_X, \phi_Y)$ such that $Z$ is a SFT and $\phi_X: Z \to X, \phi_Y: Z \to Y$ are both finite-to-one factors. Ban et al.\ \cite{BCL-JDE2012} asserted that the existence of a \emph{factor-like} matrix helps in determining whether or not there is a map between $\mathbf{Y}^{(1)}$ and $\mathbf{Y}^{(2)}$. A nonnegative $m \times n$ integral matrix $E$ is called factor-like if, for each fixed row, the summation of all entries is equal to $1$.

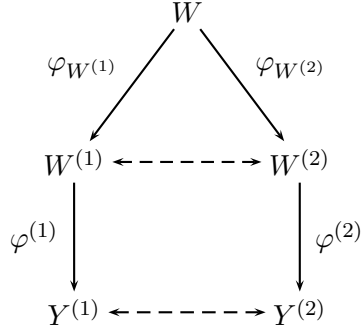
\begin{figure}
\begin{center}
\begin{pspicture}(3,4)
\psset{nodesep=0.1cm}
\rput(1.5,4){\rnode{A}{$W$}}
\rput(0,2){\rnode{B}{$W^{(1)}$}}\rput(3,2){\rnode{C}{$W^{(2)}$}}
\rput(0,0){\rnode{D}{$Y^{(1)}$}}\rput(3,0){\rnode{E}{$Y^{(2)}$}}
\ncline{->}{A}{B}\Bput{$\varphi_{W^{(1)}}$}\ncline{->}{A}{C}\Aput{$\varphi_{W^{(2)}}$}\ncline[linestyle=dashed]{<->}{B}{C}
\ncline{->}{B}{D}\Bput{$\varphi^{(1)}$}\ncline{->}{C}{E}\Aput{$\varphi^{(2)}$}
\ncline[linestyle=dashed]{<->}{D}{E}
\end{pspicture}
\end{center}
\caption{For the case that $h(\mathbf{Y}^{(1)}) = h(\mathbf{Y}^{(2)})$, we get a trianglular structure. Whether the dash lines can be replaced by solid lines infers whether the structure of the hidden and output spaces are related.}
\label{fig-Y1-Y2-triangle}
\end{figure}

\begin{proposition}[See {\cite[Proposition 3.15, Theorem 3.17]{BCL-JDE2012}}]\label{prop-factor-like-embed}
Let $T^{(i)}$ be the transition matrix of $W^{(i)}$ for $i = 1, 2$. Suppose $E$ is a factor-like matrix such that $T^{(i)}E = E T^{(\overline{i})}$, then there is a map $\pi: W^{(i)} \to W^{(\overline{i})}$ which preserves topological entropy, where $i + \overline{i} = 3$. Furthermore, if $\phi^{(i)}$ is a conjugacy, then there is a map $\overline{\pi}: \mathbf{Y}^{(i)} \to \mathbf{Y}^{(\overline{i})}$ which preserves topological entropy.
\end{proposition}

Proposition \ref{prop-factor-like-embed} infers a criterion for the existence of maps between $W^{(1)}, W^{(2)}$ and $\mathbf{Y}^{(1)}, \mathbf{Y}^{(2)}$ when $\mathbf{Y}^{(1)}$ and $\mathbf{Y}^{(2)}$ are FSE. Figure \ref{fig-Y1-Y2-triangle} illustrates a triangular structure between $W^{(i)}$ and $\mathbf{Y}^{(i)}$. The structures of the hidden and output spaces are related if the dash lines can be replaced by solid lines. Some natural questions follow immediately.

\begin{problem}\label{prob-dim-W1W2}
Suppose $\pi: W^{(i)} \to W^{(\overline{i})}$ exists.
\begin{enumerate}[a.]
\item Let $\mu$ be a Markov measure on $W^{(i)}$. Is $\pi \mu$ a Markov measure on $W^{(\overline{i})}$, where $\pi \mu := \mu \circ \pi^{-1}$ is the push-forward measure of $\mu$?

\item Suppose $\pi$ is surjective. For each Markov measure $\mu'$ on $W^{(\overline{i})}$, does there exist a $\mu$ on $W^{(i)}$ such that $\pi \mu = \mu'$?

\item How is the Hausdorff dimension $\dim W^{(\overline{i})}$ related to the Hausdorff dimension $\dim W^{(i)}$?
\end{enumerate}
\end{problem}

Problem \ref{prob-dim-W1W2} considers whether or not a topological map connects the measures and the Hausdorff dimension of two spaces. Notably, $W^{(1)}, W^{(2)}$ are topological Markov chains. It is getting more complicated when investigating the hidden and output spaces.

\begin{problem}\label{prob-dim-Y1Y2}
Suppose $\overline{\pi}: \mathbf{Y}^{(i)} \to \mathbf{Y}^{(\overline{i})}$ exists.
\begin{enumerate}[a.]
\item Let $\nu$ be a maximal measure on $\mathbf{Y}^{(i)}$. Is $\overline{\pi} \nu$ a maximal measure on $\mathbf{Y}^{(\overline{i})}$?

\item Suppose $\overline{\pi}$ is surjective. For each Markov measure $\nu'$ on $\mathbf{Y}^{(\overline{i})}$, does there exist a $\nu$ on $\mathbf{Y}^{(i)}$ such that $\overline{\pi} \nu = \nu'$?

\item Is the Hausdorff dimension $\dim \mathbf{Y}^{(\overline{i})}$ related to the Hausdorff dimension $\dim \mathbf{Y}^{(i)}$?
\end{enumerate}
\end{problem}

\subsection{Shift Spaces and Hausdorff Dimension}

In this subsection, we recall some definitions and properties of shift spaces and the Hausdorff dimension for the reader's convenience. The detailed information is referred to in \cite{LM-1995, P-1997}. Let $\mathcal{A}$ be a finite set with cardinality $|\mathcal{A}| = n$, which we consider to be an alphabet of symbols. Without the loss of generality, we usually take $\mathcal{A} = \{0, 1, \ldots, n-1\}$. The full $\mathcal{A}$-shift $\mathcal{A}^{\mathbb{Z}}$ is the collection of all bi-infinite sequences with entries from $\mathcal{A}$. More precisely,
$$
\mathcal{A}^{\mathbb{Z}} = \{\alpha = (\alpha_i)_{i \in \mathbb{Z}}: \alpha_i \in \mathcal{A} \text{ for all } i \in \mathbb{Z}\}.
$$
The shift map $\sigma$ on the full shift $\mathcal{A}^{\mathbb{Z}}$ is defined by
$$
\sigma(\alpha)_i = \alpha_{i+1} \quad \text{for} \quad i \in \mathbb{Z}.
$$
A \emph{shift space} $X$ is a subset of $\mathcal{A}^{\mathbb{Z}}$ such that $\sigma(X) \subseteq X$. $\mathcal{A}^{\mathbb{Z}}$ is a compact metric space endowed with the metric
$$
d(x, y) = \sum_{i \in \mathbb{Z}} \frac{|x_i - y_i|}{n^{|i|+1}}, \quad x, y \in \mathcal{A}^{\mathbb{Z}}.
$$

Two specific types of shift spaces that are related to our investigation are subshifts of finite type and sofic shifts. First we introduce the former. For each $k \in \mathbb{N}$, let
$$
\mathcal{A}_k = \{w_0 w_1 \cdots w_{k-1}: w_i \in \mathcal{A}, 0 \leq i \leq k-1\}
$$
denote the collection of words of length $k$ and let $\mathcal{A}_0$ denote the empty set. A \emph{cylinder} $I \subset \mathcal{A}^{\mathbb{Z}}$ is
$$
I = \{x \in \mathcal{A}^{\mathbb{Z}}: x_i x_{i+1} \cdots x_{i+k-1} = \omega_0 \omega_1 \cdots \omega_{k-1}\},
$$
for some $i \in \mathbb{Z}, k \in \mathbb{N}$, and $\omega_0 \omega_1 \cdots \omega_{k-1} \in \mathcal{A}_k$. (Sometimes we also write $I = [\omega_0, \omega_1, \cdots, \omega_{k-1}]$.) If $X$ is a shift space and there exists $L \geq 0$ and $\mathcal{F} \subseteq \cup_{0 \leq k \leq L} \mathcal{A}_k$ such that
$$
X = \{(\alpha_i)_{i \in \mathbb{Z}}: \alpha_i \alpha_{i+1} \cdots \alpha_{i+k-1} \notin \mathcal{F} \text{ for } k \leq L, i \in \mathbb{Z}\}
$$
then we say that $X$ is a SFT. The SFT is \emph{$L$-step} if words in $\mathcal{F}$ have length at most $L+1$.

Notably, it is known that, without the loss of generality, SFTs can be defined by $0, 1$ transition matrices. For instance, let $T$ be an $n \times n$ matrix with rows and columns indexed by $\mathcal{A}$ and entries from $\{0, 1\}$. Then
$$
X = \{x \in \mathcal{A}^{\mathbb{Z}}: T(x_i, x_{i+1}) = 1 \text{ for all } i \in \mathbb{Z}\}
$$
is a one-step SFT. (It is also known as a \emph{topological Markov chain} by Parry.) A topological Markov chain is called irreducible/mixing if its transition matrix is irreducible/mixing.

An extended concept of SFTs is called \emph{sofic shifts}. A sofic shift is a subshift which is the image of a SFT under a factor map. Suppose $X$ and $Y$ are two shift spaces. A \emph{factor map} is a continuous onto map $\pi: X \to Y$ such that $\pi \circ \sigma_{X} = \sigma_{Y} \circ \pi$. A one-to-one factor map is called a topological conjugacy. A sofic shift is \emph{irreducible} if it is the image of an irreducible SFT.

In the previous subsection we mentioned that the topological entropy illustrates the complexity of the topological behavior of a system. Aside from the topological entropy, the Hausdorff dimension characterizes its geometrical structure. The concept of the Hausdorff dimension generalizes the notion of the dimension of a real vector space and helps to distinguish the difference of measure zero sets. We recall the definition of the Hausdorff dimension for reader's convenience.

Given $\epsilon > 0$, an $\epsilon$-cover $\{U_i\}$ of $X$ is a cover such that the diameter of $U_i$ is less than $\epsilon$ for each $i$. Putting
\begin{equation}
\mathcal{H}^s(X) = \liminf_{\epsilon \to 0} \sum_{i=1}^{\infty} \delta(U_i)^s,
\end{equation}
where $\delta(U_i)$ denotes the diameter of $U_i$. The Hausdorff dimension of $X$ is defined by
\begin{equation}
\dim X = \inf \{s: \mathcal{H}^s(X) = 0\}.
\end{equation}

For subsets that are invariant under a dynamical system we can pose the problem of the Hausdorff dimension of an invariant measure. To be precise let us consider a map $g: X \to X$ with invariant probability measure $\mu$. The stochastic properties of $g$ are related to the topological structure of $X$. A relevant quantitative characteristic, which can be used to describe the  complexity of the topological structure of $X$, is the Hausdorff dimension of the measure $\mu$. The Hausdorff dimension of a probability measure $\mu$ on $X$ is defined by
$$
\dim \mu = \inf \{\dim Z: Z \subset X \text{ and } \mu(Z) = 1\}.
$$
$\mu$ is called a \emph{measure of full Hausdorff dimension} (MFHD) if $\dim \mu = \dim X$. A MFHD is used for the investigation of the Hausdorff dimension $\dim X$, and the computation of the Hausdorff dimension of a MFHD corresponds to the computation of the \emph{measure-theoretic entropy}, an analogous quantity as the topological entropy that illustrates the complexity of a physical system, of $X$ with respect to the MFHD \cite{Bow-PMI1979, P-1997}. This causes the discussion of measure-theoretic entropy to play an important role in this elucidation.

Given a shift space $(X, \sigma)$, we denote by $\mathcal{M}(X)$ the set of $\sigma$-invariant Borel Probability measures on $X$. Suppose $P$ is a irreducible stochastic matrix and $p$ a stochastic row vector such that $pP = p$, that is, the summation of entries in each row of $P$ is $1$, and the summation of the entries of $p$ is $1$. Notably such $p$ is unique due to the irreducibility of $P$. Define a $0, 1$ matrix $T$ by $T(i, j) = 1$ if and only if $P(i, j) > 0$. (The matrix $T$ is sometimes known as the \emph{incidence matrix} of $P$.) Denote the space of \emph{right-sided} SFT $X_T^+$ by
$$
X_T^+ = \{x \in \mathcal{A}^{\mathbb{Z}^+}: T(x_i, x_{i+1}) = 1 \text{ for all } i \geq 0\}.
$$
It is seen that $X_T^+$ is embedded as a subspace of $X_T$. The metric on $X_T^+$ is endowed with
$$
d^+(x, y) = \sum_{i \geq 0} \frac{|x_i - y_i|}{n^{i+1}}, \quad x, y \in X_T^+.
$$
Then $(p, P)$ defines an invariant measure $\mu^+$ on $X_T^+$ as
$$
\mu^+([\omega_0, \omega_1, \cdots, \omega_{k-1}]) = p(\omega_0) P(\omega_0, \omega_1) \cdots P(\omega_{k-2}, \omega_{k-1})
$$
for each cylinder set $I^+ = [\omega_0, \omega_1, \cdots, \omega_{k-1}] \subset X_T^+$ by the Kolmogorov Extension Theorem. Moreover, a measure $\mu^+$ on $X_T^+$ is Markov if and only if it is determined by a pair $(p, P)$ as above.

Similar to the above, we define the \emph{left-sided} SFT $X_T^-$ by
$$
X_T^- = \{x \in \mathcal{A}^{\mathbb{Z}^-}: T(x_{-i}, x_{-i+1}) = 1 \text{ for all } i \in \mathbb{N}\}.
$$
Then $X_T^-$ is a subspace of $X_T$, and the metric on $X_T^-$ is endowed with
$$
d^-(x, y) = \sum_{i \leq 0} \frac{|x_i - y_i|}{n^{-i+1}}, \quad x, y \in X_T^-.
$$
Let $Q$ be the transpose of $P$ and $q$ is the stochastic row vector such that $qQ = q$. Then $(q, Q)$ defines an invariant measure $\mu^-$ on $X_T^-$ as
$$
\mu^-([\omega_{-k+1}, \cdots, \omega_{-1}, \omega_0]) = q(\omega_0) Q(\omega_0, \omega_{-1}) \cdots Q(\omega_{-k+2}, \omega_{-k+1})
$$
for each cylinder set $I^- = [\omega_{-k+1}, \cdots, \omega_{-1}, \omega_0] \subset X_T^-$.

Notice that given a cylinder $I = [\omega_{-\ell+1}, \ldots, \omega_0, \ldots, \omega_{k-1}] \subset X_T$, $\ell, k \geq 1$, $I$ can be identified with the direct product $I^+ \times I^-$, where $I^+ = [\omega_0, \cdots, \omega_{k-1}] \subset X_T^+$ and $I^- = [\omega_{-\ell+1}, \cdots, \omega_0] \subset X_T^-$. Furthermore, $(p, P)$ defines an invariant measure $\mu$ on $X_T$ as $\mu(I) \approx \mu^+(I^+) \mu^-(I^-)$ for any cylinder $I \subset X_T$. To be precise, there exist positive constants $A_1$ and $A_2$ such that for integers $k, \ell \geq 0$, and any cylinder $I = [\omega_{-\ell+1}, \ldots, \omega_0, \ldots, \omega_{k-1}] \subset X_T$,
\begin{equation}\label{eq-mu-equal-mu+mu-}
A_1 \leq \dfrac{\mu(I)}{\mu^+(I^+) \mu^-(I^-)} \leq A_2.
\end{equation}
Combining \eqref{eq-mu-equal-mu+mu-} with the fact that every cylinder $I \in X_T$ is identified with $I^+ \times I^-$ infers that the study of the measure-theoretic entropy of one-sided subspace $X_T^+$/$X_T^-$ is significant for investigating the measure-theoretic entropy of $X_T$. What is more, the computation of the Hausdorff dimension of $X$ is closely related to the computation of the Hausdorff dimension of $X_T^+$/$X_T^-$. The reader is referred to \cite{Kit-1998} for more details.

Now we are ready to introduce the general definition of the measure-theoretic entropy. Given a shift space $X \subseteq \mathcal{A}^{\mathbb{Z}}$ and an invariant probability measure $\mu$ on $X$, the measure-theoretic entropy of $X$ with respect to $\mu$ is given by
$$
h_{\mu}(X) = - \lim_{n \to \infty} \frac{1}{n} \sum_{I \in X_n} \mu(I) \log \mu(I),
$$
where $X_n$ denotes the collection of cylinders of length $n$ in $X$. The concepts of the measure-theoretic and topological entropies are connected by the \emph{Variational Principle}:
$$
h(X) = \sup \{h_{\mu}(X): \mu \in \mathcal{M}(X)\}.
$$
$\mu$ is called a \emph{measure of maximal entropy} (also known as maximal measure) if $h_{\mu}(X) = h(X)$. Notably, suppose $X_T^+$/$X_T^-$/$X_T$ is a SFT determined by $T$, which is the incidence matrix of an irreducible stochastic matrix $P$. It is well-known that the Markov measure $\mu^+$/$\mu^-$/$\mu$, derived from the pair $(p, P)$, is the \emph{unique} measure of maximal entropy.

\subsection{Results}

This subsection is devoted to illustrating the main results of the present elucidation. First we recall a well-known result.

\begin{theorem}[See {\cite[Theorem 4.1.7]{Kit-1998}}]\label{thm-diamond-inf2one}
Suppose $\phi: X \to Y$ is a one-block factor map between mixing SFTs, and $X$ has positive entropy. Then either
\begin{enumerate}
\item $\phi$ is uniformly bounded-to-one,

\item $\phi$ has no diamond,

\item $h(X)=h(Y)$
\end{enumerate}
or
\begin{enumerate}\setcounter{enumi}{3}
\item $\phi$ is uncountable-to-one on some point,

\item $\phi$ has diamond,

\item $h(X)>h(Y)$.
\end{enumerate}
\end{theorem}

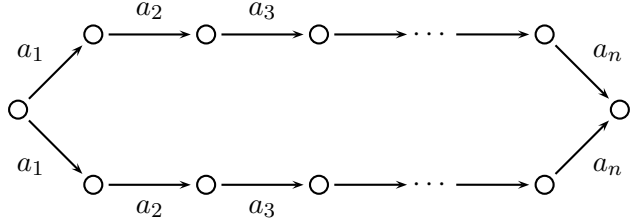
\begin{figure}
\begin{center}
\begin{pspicture}(13,4)\psset{nodesep=2pt}
\rput(1.5,2){\ovalnode{A}{\Tn}}\rput(9.5,2){\ovalnode{B}{\Tn}}
\rput(2.5,3){\ovalnode{C}{\Tn}}\rput(2.5,1){\ovalnode{D}{\Tn}}
\rput(4,3){\ovalnode{E}{\Tn}}\rput(4,1){\ovalnode{F}{\Tn}}
\rput(5.5,3){\ovalnode{G}{\Tn}}\rput(5.5,1){\ovalnode{H}{\Tn}}
\rput(7,3){\rnode{I}{\Tn{$\cdots$}}}\rput(7,1){\rnode{J}{\Tn{$\cdots$}}}
\rput(8.5,3){\ovalnode{K}{\Tn}}\rput(8.5,1){\ovalnode{L}{\Tn}}
\ncline{->}{A}{C}\naput{$a_1$}\ncline{->}{C}{E}\naput{$a_2$}\ncline{->}{E}{G}\naput{$a_3$}
\ncline{->}{G}{I}\ncline{->}{I}{K}\ncline{->}{K}{B}\naput{$a_n$}
\ncline{->}{A}{D}\nbput{$a_1$}\ncline{->}{D}{F}\nbput{$a_2$}\ncline{->}{F}{H}\nbput{$a_3$}
\ncline{->}{H}{J}\ncline{->}{J}{L}\ncline{->}{L}{B}\nbput{$a_n$}
\end{pspicture}
\end{center}
\caption{A factor map $\phi: X \to Y$ has a diamond infers that there exists a pair of distinct points in $X$ differing in only finitely many coordinates with the same image under $\phi$. It is named after the shape of its labeled graph representation.}
\label{fig-diamond}
\end{figure}

A \emph{diamond} for $\phi: X \to Y$ is a pair of distinct points in $X$ differing in only a finite number of coordinates with the same image under $\phi$ (cf.~Figure \ref{fig-diamond}). Theorem \ref{thm-diamond-inf2one} reveals that the investigation of the  existence of diamonds is equivalent to the study of infinite-to-one factor maps.

Without the loss of generality, we may assume that every factor map $\phi$ is a one-block code. That is, there exists $\Phi: \mathcal{A}(X) \to \mathcal{A}(Y)$ such that $\phi(x)_i = \Phi(x_i)$ for $i \in \mathbb{Z}$. Theorem \ref{thm-diamond-inf2one}, in other words, indicates that every factor map is either finite-to-one or infinite-to-one. In \cite{BCL-JDE2012}, the authors investigated those finite-to-one factor maps. The infinite-to-one factor maps are examined in this study. Once a factor map exists, we can use it to formulate the Hausdorff dimension of these spaces.

We start with considering the case that $\mathbf{Y}^{(1)}$ is finitely shift equivalent to $\mathbf{Y}^{(2)}$. Two spaces are FSE infers that a factor map between them, if it exists, is finite-to-one. Let $X$ be a shift space. A point $x \in X$ is said to be \emph{doubly transitive} if, for every $k \in \mathbb{N}$ and word $w$ in $X$, there exist $\overline{\ell} < 0 < \ell$ with $|\overline{\ell}|, \ell > k$ such that
$$
x_{\overline{\ell}-|w|+1} \cdots x_{\overline{\ell}} = w \quad \text{and} \quad x_{\ell} \cdots x_{\ell + |w|-1} = w.
$$
Suppose $\phi: X \to Y$ is a factor map. If there is a positive integer $K$ such that every doubly transitive point of $Y$ has exactly $K$ preimages under $\phi$. Such $K$ is called the \emph{degree} of $\phi$ and we define $d_{\phi} = K$ \cite{LM-1995}.

Let $w = w_1 \cdots w_n$ be a word in $Y$. For $1 \leq i \leq n$, define $d^*_{\phi}(w, i)$ to be the number of alphabets at coordinate $i$ in the preimages of $w$. In other words,
$$
d^*_{\phi}(w, i) = \# \{u_i \in \mathcal{A}(X): u = u_1 \cdots u_n \in X_n, \Phi(u_i) = w_i \text{ for } 1 \leq i \leq n\}.
$$
Denote
$$
d^*_{\phi} = \min \{d^*_{\phi}(w, i): w \in B(Y), 1 \leq i \leq |w|\},
$$
where $B(Y)$ indicates the collection of words in $Y$.

\begin{definition}
We say that $w \in B(Y)$ is a \emph{magic word} if $d^*_{\phi}(w, i) = d^*_{\phi}$ for some $i$. Such an index $i$ is called a \emph{magic coordinate}.
\end{definition}

We say a factor map $\phi$ has a \emph{synchronizing word} if there is a finite block $y_1 y_2 \cdots y_n \in \mathcal{A}_n(Y)$ such that, each element in $\phi^{-1}(y_1 y_2 \cdots y_n)$ admits the same terminal entry. A finite-to-one factor map $\phi$ has a synchronizing word indicating that the push-forward measure of a measure of maximal entropy under a finite-to-one factor map is still a measure of maximal entropy. The following is our first main result.

\begin{theorem}\label{main-thm-FSE}
Suppose the hidden space $\mathbf{Y}^{(1)}$ and the output space $\mathbf{Y}^{(2)}$ are FSE. Let $W^{(i)}$ be irreducible with finite-to-one factor map $\phi^{(i)}: W^{(i)} \to \mathbf{Y}^{(i)}$ for $i = 1, 2$. If $\phi^{(i)}$ has a synchronizing word, then
\begin{enumerate}[\bf i)]
  \item There is a one-to-one correspondence between $\mathcal{M}_{\max}(W^{(i)})$ and $\mathcal{M}_{\max}(\mathbf{Y}^{(i)})$, where $\mathcal{M}_{\max}(X)$ indicates the set of measures of maximal entropy.

  \item Let $m_i = |\mathcal{A}(W^{(i)})|, n_i = |\mathcal{A}(\mathbf{Y}^{(i)})|$, and $\mu^{(i)}$ a maximal measure of $W^{(i)}$. Then
  \begin{align*}
  \dim W^{(i)} &= \frac{h_{\mu^{(i)}}(W^{(i)})}{\log m_i} = 2 \frac{h_{\mu^{(i),+}}(W^{(i),+})}{\log m_i} \\
  \intertext{and}
  \dim \mathbf{Y}^{(i)} &=  \frac{h_{\phi^{(i)}\mu^{(i)}}(\mathbf{Y}^{(i)})}{\log n_i} =  \frac{h_{\mu^{(i)}}(W^{(i)})}{\log n_i} = 2  \frac{h_{\mu^{(i),+}}(W^{(i),+})}{\log n_i},
  \end{align*}
  where $\mu^{(i),+}$ is the maximal measure of the right-sided subspace $W^{(i),+}$ of $W^{(i)}$.

  \item Suppose $\pi: W^{(i)} \to W^{(\overline{i})}$ is a factor map and $\nu^{(i)} = \phi^{(i)} \mu^{(i)}$ for $i = 1, 2$, where $i + \overline{i} = 3$. If
  $$
  \dim \mathbf{Y}^{(i)} = \frac{h_{\nu^{(i)}}(\mathbf{Y}^{(i)})}{\log n_i},
  $$
  then
  $$
  \dim \mathbf{Y}^{(\overline{i})} = \frac{h_{\overline{\pi}\nu^{(i)}}(\mathbf{Y}^{(\overline{i})})}{\log n_{\overline{i}}} = \frac{h_{\nu^{(\overline{i})}}(\mathbf{Y}^{(\overline{i})})}{\log n_{\overline{i}}}
  $$
  for some $\overline{\pi}$.
\end{enumerate}
\end{theorem}

In contrast with the map connecting $\overline{\pi}: \mathbf{Y}^{(i)} \to \mathbf{Y}^{(\overline{i})}$, if it exists, being finite-to-one when $h(\mathbf{Y}^{(1)}) = h(\mathbf{Y}^{(2)})$, Theorem \ref{thm-diamond-inf2one} indicates that $\overline{\pi}$ must be infinite-to-one for the case where $h(\mathbf{Y}^{(1)}) \neq h(\mathbf{Y}^{(2)})$. Intuitively, the number of infinite-to-one factor maps is much larger than the number of finite-to-one factor maps. Computer assisted examination serves affirmative results for MCNN \cite{BC-AMC2013}.

Suppose $\phi: X \to Y$ is a factor map and $h(X) \neq h(Y)$. Intuitively there is a maximal measure in $Y$ with infinite preimage. It is natural to ask whether these preimages are isomorphic to one another. The isomorphism of two measures demonstrates their measure-theoretic entropies coincide with the same value. In \cite{BT-TAMS1984}, Boyle and Tuncel indicated that any two Markov measures associated with the same image are isomorphic to each other if $\phi$ is a \emph{uniform factor}. We say that $\phi$ is uniform if $\phi \mu \in \mathcal{M}_{\max}(Y)$ for every $\mu \in \mathcal{M}_{\max}(X)$. $\phi$ is a uniform factor indicating $\dim Y$ is related to $\dim X$.

\begin{theorem}\label{main-thm-ITO}
Assume that $h(\mathbf{Y}^{(1)}) \neq h(\mathbf{Y}^{(2)})$. Let $W^{(i)}$ be irreducible with finite-to-one factor map $\phi^{(i)}: W^{(i)} \to \mathbf{Y}^{(i)}$ for $i = 1, 2$.
\begin{enumerate}[\bf i)]
  \item Suppose $\pi: W^{(i)} \to W^{(\overline{i})}$ is a uniform factor. Let $m_i = |\mathcal{A}(W^{(i)})|, n_i = |\mathcal{A}(\mathbf{Y}^{(i)})|$, and $\mu^{(i)}$ be a maximal measure of $W^{(i)}$. If
  $$
  \dim W^{(i)} = \frac{h_{\mu^{(i)}}(W^{(i)})}{\log m_i},
  $$
  then
  $$
  \dim W^{(\overline{i})} = \frac{h_{\mu^{(\overline{i})}}(W^{(\overline{i})})}{\log m_{\overline{i}}} = \frac{h_{\pi\mu^{(i)}}(W^{(\overline{i})})}{\log m_{\overline{i}}}.
  $$
\end{enumerate}
Furthermore, suppose $h(\mathbf{Y}^{(i)}) > h(\mathbf{Y}^{(\overline{i})})$ and $\phi^{(i)}$ has a synchronizing word, then
\begin{enumerate}[\bf i)]\setcounter{enumi}{1}
  \item There exists a factor map $\overline{\pi}: \mathcal{M}_{\max}(\mathbf{Y}^{(i)}) \to \mathcal{M}_{\max}(\mathbf{Y}^{(\overline{i})})$.

  \item If
  $$
  \dim \mathbf{Y}^{(i)} = \dfrac{h_{\nu^{(i)}}(\mathbf{Y}^{(i)})}{\log n_i},
  $$
  then
  $$
  \dim \mathbf{Y}^{(\overline{i})} = \dfrac{h_{\overline{\pi} \nu^{(i)}}(\mathbf{Y}^{(\overline{i})})}{\log n_{\overline{i}}}.
  $$
\end{enumerate}
\end{theorem}

We postpone the proof of Theorems \ref{main-thm-FSE} and \ref{main-thm-ITO} to the following section. In the mean time, we will introduce the factor maps between the solution, hidden, and output spaces.

\section{Existence of Factors}

The existence of factor maps plays an important role in the proof of Theorems \ref{main-thm-FSE} and \ref{main-thm-ITO}. First we focus on whether or not a factor map between two spaces exists, and, if it exists, the possibility of finding out an explicit form.

\subsection{Classification of Solution Spaces}

To clarify the discussion, we consider a simplified case. A simplified MCNN (SMCNN) is unveiled as
\begin{equation}\label{eq-2layer-mcnn-aar}
\left\{
\begin{split}
\frac{d x_i^{(1)}}{dt} &= - x_i^{(1)} + a^{(1)} y_i^{(1)} + a^{(1)}_r y_{i+1}^{(1)} + z^{(1)}, \\
\frac{d x_i^{(2)}}{dt} &= - x_i^{(2)} + a^{(2)} y_i^{(2)} + a^{(2)}_r y_{i+1}^{(2)} + b^{(2)} u_i^{(2)} + b^{(2)}_r u_{i+1}^{(2)} + z^{(2)}.
\end{split}
\right.
\end{equation}
Suppose $\mathbf{y} = \left({\cdots y_{-1}^{(2)} y_0^{(2)} y_1^{(2)} \cdots \atop \cdots y_{-1}^{(1)} y_0^{(1)} y_1^{(1)} \cdots}\right)$ is a mosaic pattern. For $i \in \mathbb{Z}$, $y_i^{(1)} = 1$ if and only if
\begin{equation}
a^{(1)} + z^{(1)} - 1 > -a^{(1)}_r y^{(1)}_{i+1}.\label{eq-2layer-ineq1}
\end{equation}
Similarly, $y_i^{(1)} = -1$ if and only if
\begin{equation}
a^{(1)} - z^{(1)} - 1 > a^{(1)}_r y^{(1)}_{i+1}.\label{eq-2layer-ineq2}
\end{equation}
The same argument asserts
\begin{align}
a^{(2)} + z^{(2)} - 1 &> -a^{(2)}_r y^{(2)}_{i+1} - (b^{(2)} u^{(2)}_i + b^{(2)}_r u^{(2)}_{i+1}),\label{eq-2layer-ineq3} \\
\intertext{and}
a^{(2)} - z^{(2)} - 1 &> a^{(2)}_r y^{(2)}_{i+1} + (b^{(2)} u^{(2)}_i + b^{(2)}_r u^{(2)}_{i+1})\label{eq-2layer-ineq4}
\end{align}
are the necessary and sufficient conditions for $y_i^{(2)} = -1$ and $y_i^{(2)} = 1$, respectively. Note that the quantity $u^{(2)}_i$ in \eqref{eq-2layer-ineq3} and \eqref{eq-2layer-ineq4} satisfies $|u^{(2)}_i| = 1$ for each $i$. Define $\xi_1: \{-1, 1\} \to \mathbb{R}$ and $\xi_2: \{-1, 1\}^{\mathbb{Z}_{3 \times 1}} \to \mathbb{R}$ by
$$
\xi_1(w) = a^{(1)}_r w, \qquad \xi_2(w_1, w_2, w_3) = a^{(2)}_r w_1 + b^{(2)} w_2 + b^{(2)}_r w_3.
$$
Set
\begin{align*}
\mathcal{B}^{(1)} &= \left\{\boxed{y^{(1)} y_r^{(1)}}: y^{(1)}, y_r^{(1)} \in \{-1, 1\} \text{ satisfy } \eqref{eq-2layer-ineq1}, (\ref{eq-2layer-ineq2})\right\}, \\
\mathcal{B}^{(2)} &= \left\{\boxed{y^{(2)} y_r^{(2)} \atop \displaystyle u^{(2)} u_r^{(2)}}: y^{(2)}, y_r^{(2)}, u^{(2)}, u_r^{(2)} \in \{-1, 1\} \text{ satisfy } \eqref{eq-2layer-ineq3}, \eqref{eq-2layer-ineq4}\right\}.
\end{align*}
That is,
\begin{align*}
\boxed{y^{(1)} y_r^{(1)}} \in \mathcal{B}^{(1)} &\Leftrightarrow \left\{
                                              \begin{array}{ll}
                                                a^{(1)} + z^{(1)} - 1 > - \xi_1(y_r^{(1)}), & \hbox{if $y^{(1)}=1$;} \\
                                                a^{(1)} - z^{(1)} - 1 > \xi_1(y_r^{(1)}), & \hbox{if $y^{(1)}=-1$.}
                                              \end{array}
                                            \right. \\
\boxed{y^{(2)} y_r^{(2)} \atop \displaystyle u^{(2)} u_r^{(2)}} \in \mathcal{B}^{(2)} &\Leftrightarrow \left\{
                                                                                        \begin{array}{ll}
                                                                                          a^{(2)} + z^{(2)} - 1 > - \xi_2(y_r^{(2)}, u^{(2)}, u_r^{(2)}), & \hbox{if $y^{(2)}=1$;} \\
                                                                                          a^{(2)} - z^{(2)} - 1 > \xi_2(y_r^{(2)}, u^{(2)}, u_r^{(2)}), & \hbox{if $y^{(2)}=-1$.}
                                                                                        \end{array}
                                                                                      \right.
\end{align*}
The set of admissible local patterns $\mathcal{B}$ of (\ref{eq-2layer-mcnn-aar}) is then
$$
\mathcal{B} = \left\{\boxed{y y_r \atop \displaystyle u u_r}: \boxed{y y_r \atop \displaystyle u u_r} \in \mathcal{B}^{(2)} \text{ and } \boxed{u u_r} \in \mathcal{B}^{(1)}\right\}
$$
The authors indicated in \cite{BCL-JDE2012} that there exists $139,968$ regions in the parameter space $\mathcal{P}$ of SMCNNs such that any two sets of templates that are located in the same region infer the same solution spaces. The partition of the parameter space is determined as follows.

Since $y^{(1)}, y^{(2)}, y_r^{(1)}, y_r^{(2)}, u^{(2)}, u_r^{(2)} \in \{-1, 1\}$, $a^{(1)} + z^{(1)} - 1 = - \xi_1(y_r^{(1)})$ and $a^{(1)} + z^{(1)} - 1 = \xi_1(y_r^{(1)})$ partition $a^{(1)}\ z^{(1)}$-plane into $9$ regions, the ``order" of lines $a^{(1)} + z^{(1)} - 1 = (-1)^{\ell} \xi_1(y_r^{(1)})$, $\ell = 0, 1$, comes from the sign of $a_r^{(1)}$. Thus the parameter space $\{(a^{(1)}, a^{(1)}_r, z^{(1)})\}$ is partitioned into $2 \times 9 = 18$ regions. Similarly, $a^{(2)} + z^{(2)} - 1 > - \xi_2(y_r^{(2)}, u^{(2)}, u_r^{(2)})$ and $a^{(2)} + z^{(2)} - 1 > \xi_2(y_r^{(2)}, u^{(2)}, u_r^{(2)})$ partition $a^{(2)}\ z^{(2)}$-plane into $81$ regions. The order of $a^{(2)} + z^{(2)} - 1 > (-1)^{\ell} \xi_2(y_r^{(2)}, u^{(2)}, u_r^{(2)})$ can be uniquely determined according to the following procedures.
\begin{enumerate}[(i)]
  \item The signs of $a_r^{(2)}, b^{(2)}, b_r^{(2)}$.
  \item The magnitude of $a_r^{(2)}, b^{(2)}, b_r^{(2)}$.
  \item The competition between the parameters with the largest magnitude and the others. In other words, suppose $m_1 > m_2 > m_3$ represent $|a_r^{(2)}|, |b^{(2)}|, |b_r^{(2)}|$. We need to determine whether $m_1 > m_2 + m_3$ or $m_1 < m_2 + m_3$.
\end{enumerate}
This partitions the parameter space $\{(a^{(2)}, a^{(2)}_r, b^{(2)}, b^{(2)}_r, z^{(2)})\}$ into $8 \times 6 \times 2 \times 81 = 7776$ regions. Hence the parameter space $\mathcal{P}$ is partitioned into $81 \times 7776 = 139,968$ equivalent subregions.

Since the solution space $\mathbf{Y}$ is determined by the basic set of admissible local patterns, these local patterns play an essential role for investigating SMCNNs. Substitute mosaic patterns $-1$ and $1$ as symbols $-$ and $+$, respectively. Define the ordering matrix of $\{-, +\}^{\mathbb{Z}_{2 \times 2}}$ by
$$
\mathbb{X} = \bordermatrix{
  & \fbox{$\overset{\displaystyle-}{-}$} & \fbox{$\overset{\displaystyle-}{+}$} & \fbox{$\overset{\displaystyle+}{-}$} & \fbox{$\overset{\displaystyle+}{+}$} \vspace{1mm}\cr
\fbox{$\overset{\displaystyle-}{-}$} & \fbox{$\overset{\displaystyle--}{--}$} & \fbox{$\overset{\displaystyle--}{-+}$} & \fbox{$\overset{\displaystyle-+}{--}$} & \fbox{$\overset{\displaystyle-+}{-+}$} \vspace{1mm}\cr
\fbox{$\overset{\displaystyle-}{+}$} &\fbox{$\overset{\displaystyle--}{+-}$} & \fbox{$\overset{\displaystyle--}{++}$} & \fbox{$\overset{\displaystyle-+}{+-}$} & \fbox{$\overset{\displaystyle-+}{++}$} \vspace{1mm}\cr
\fbox{$\overset{\displaystyle+}{-}$} &\fbox{$\overset{\displaystyle+-}{--}$} & \fbox{$\overset{\displaystyle+-}{-+}$} & \fbox{$\overset{\displaystyle++}{--}$} & \fbox{$\overset{\displaystyle++}{-+}$} \vspace{1mm}\cr
\fbox{$\overset{\displaystyle+}{+}$} &\fbox{$\overset{\displaystyle+-}{+-}$} & \fbox{$\overset{\displaystyle+-}{++}$} & \fbox{$\overset{\displaystyle++}{+-}$} & \fbox{$\overset{\displaystyle++}{++}$}}
= (x_{pq})_{1 \leq p, q \leq 4}
$$
Notably each entry in $\mathbb{X}$ is a $2 \times 2$ pattern since $\mathcal{B}$ consists of $2 \times 2$ local patterns. Suppose that $\mathcal{B}$ is given. The transition matrix $T \equiv T(\mathcal{B})  \in \mathcal{M}_4(\{0, 1\})$ is defined by
$$
T(p,q) = \left\{
           \begin{array}{ll}
             1, & \hbox{if $x_{pq} \in \mathcal{B}$;} \\
             0, & \hbox{otherwise.}
           \end{array}
         \right.
$$

Let $\mathcal{A} = \{\alpha_0, \alpha_1, \alpha_2, \alpha_3\}$, where
$$
\alpha_0 = --, \quad \alpha_1 = -+, \quad \alpha_2 = +- \quad \text{and} \quad \alpha_3 = ++
$$
Define $\mathcal{L}^{(1)}, \mathcal{L}^{(2)}$ by
$$
\mathcal{L}^{(1)}(y_0 y_1 \diamond u_0 u_1) = u_0 u_1 \quad \text{and} \quad \mathcal{L}^{(2)}(y_0 y_1 \diamond u_0 u_1) = y_0 y_1
$$
respectively. It is known that $T$ determines a graph while $(T, \mathcal{L}^{(i)})$ determines a labeled graph for $i = 1, 2$. As we mentioned in last section that the transition matrix $T$ determines the solution space $\mathbf{Y}$, $T$ not describe the hidden and output spaces $\mathbf{Y}^{(1)}$ and $\mathbf{Y}^{(2)}$, though. Instead, $\mathbf{Y}^{(1)}, \mathbf{Y}^{(2)}$ are illustrated by the symbolic transition matrices. The symbolic transition matrix $S^{(i)}$ is defined by
\begin{equation}\label{eq-symbolic-matrix}
S^{(i)}(p,q) = \left\{
           \begin{array}{ll}
             \alpha_j, & \hbox{if $T(p,q)=1$ and $\mathcal{L}^{(i)}(x_{pq}) = \alpha_j$ for some $j$;} \\
             \varnothing, & \hbox{otherwise.}
           \end{array}
         \right.
\end{equation}
Herein $\varnothing$ means there exists no local pattern in $\mathcal{B}$ related to its corresponding entry in the ordering matrix. A labeled graph is called \emph{right-resolving} if, for every fixed row of its symbolic transition matrix, the multiplicity of each symbol is $1$. With a little abuse of notations, $\mathbf{Y}^{(i)}$ can be described by $S^{(i)}$ which is right-resolving for $i = 1, 2$. Let $T^{(i)}$ be the incidence matrix of $S^{(i)}$, that is, $T^{(i)}$ is of the same size of $S^{(i)}$ and is defined by
$$
T^{(i)}(p,q) = \left\{
           \begin{array}{ll}
             1, & \hbox{if $S^{(i)}(p,q) \neq \varnothing$;} \\
             0, & \hbox{otherwise.}
           \end{array}
         \right.
$$
Then $W^{(i)}$ is determined by $T^{(i)}$ for $i = 1, 2$. The reader is referred to \cite{BCL-JDE2012, BCLL-JDE2009} for more details.

\subsection{Sofic Measures and Linear Representable Measures}

Theorems \ref{main-thm-FSE} and \ref{main-thm-ITO} investigate the Hausdorff dimension of $W^{(i)}$ and $\mathbf{Y}^{(i)}$ and see if they are related. The proof relies on two essential ingredients: the existence of maximal measures and factor maps. The upcoming subsection involves the former while the latter is discussed in the next two subsections separately.

Let $X$ and $Y$ be subshifts and $\phi: X \rightarrow Y$ be a factor map. Suppose $\mu$ is a Markov measure on $X$, then $\phi \mu $ is called a \emph{sofic measure} (also known as a \emph{hidden Markov measure}, cf.~\cite{BP-2011}). Let $B \in \mathbb{R}^{m \times m}$ be an irreducible matrix with spectral radius $\rho_{B}$ and positive right eigenvector $r$; the \emph{stochasticization} of $B$ is the stochastic matrix
\begin{equation*}
\mathbb{B}:= stoch(B)=\frac{1}{\rho_B }D^{-1}BD,
\end{equation*}
where $D$ is the diagonal matrix with diagonal entries $D(i,i)=r(i)$. A measure $\mu$ on $X$ is called \emph{linear representable} with dimension $m$ if there exists a triple $(x,P,y)$ with $x$ being a $1\times m$ row vector, $y$ being a $m \times 1$ column
vector and $P=(P_{i})_{i\in \mathcal{A}(X)}$, where $P_{i} \in \mathbb{R}^{m \times m}$ such that for all $I=[i_{0},\ldots ,i_{n-1}] \in
X_{n} $, the measure $\mu$ can be characterized as the following form:
\begin{equation*}
\mu (\left[ I\right] )=xP_{I}y = x P_{i_{0}}P_{i_{1}} \cdots P_{i_{n-1}} y.
\end{equation*}%
The triple $(x,P,y)$ is called the \emph{linear representation} of the measure $\mu$. The reader is referred to \cite{BP-2011} for more details.

\begin{proposition}[See {\cite[Theorem 4.20]{BP-2011}}]\label{prop-LR-pushforward-LR}
Let $X$ be an irreducible SFT with transition matrix $T \in \mathbb{R}^{m \times m}$ and $\phi: X \rightarrow Y$ be a one-block factor map. Let $\mathbb{T}=stoch(T)$ and $l$ be the probability left eigenvector of $\mathbb{T}$. Then
\begin{enumerate}[(i)]
\item The Markov measure $\mu$ on $X$ is the linear representable measure with respect to the triple $(l, P, \mathbf{1}_m)$, where $\mathbf{1}_m$ is the column vector with each entry being $1$ and $P = (\mathbb{T}_{i}) _{i\in \mathcal{A}(X)}$ for which
\begin{equation*}
P_{I} = \mathbb{T}_{i_{0}}\cdots \mathbb{T}_{i_{n-1}}, \quad \text{for all} \quad I=[i_{0},\ldots ,i_{n-1}] \in X_{n}
\end{equation*}%
here $\mathbb{T}_{k}(i,j)=\mathbb{T}( i,j) $ if $j=k$ and $\mathbb{T}_{k}(i,j)=0$ otherwise.

\item The push-forward measure $\nu =\phi \mu$ is linear representable with respect to the triple $(l,Q,\mathbf{1}_m)$, where $Q$ is generated by $(Q_{j}) _{j\in \mathcal{A}(Y)}=\left( \mathbb{T}_{j}\right) _{j\in \mathcal{A}(Y)}$ for which $\mathbb{T}_{k}(u,v)=\mathbb{T}(u,v) $ if $\phi(v)=k$ and $\mathbb{T}_{k}(u,v)=0$ otherwise.
\end{enumerate}
\end{proposition}

In the following we propose a criterion to determine whether a sofic measure is actually a Markov measure. The procedure of the criterion is systematic and is checkable which makes our method practical. Suppose the factor map $\phi: X \to Y$ is a one-block code. For $j \in \mathcal{A}(Y)$, define
$$
E_{j} = \{i: \phi(i) = j\} \quad \text{and} \quad e_j = \# E_j
$$
For each $j_1 j_2 \in Y_2$, let $N_{j_1 j_2} \in \mathbb{R}^{e_{j_1} \times e_{j_2}}$ be defined by
$$
N_{j_1 j_2}(p, q) = \left\{
\begin{array}{ll}
1, & p q \in X_2 \hbox{;} \\
0, & \hbox{otherwise.}
\end{array}\right.
$$
where $p \in E_{j_1}, q \in E_{j_2}$. Set $N = (N_{j_1 j_2})$ if $e_{j_1} = e_{j_2}$ for all $j_1 j_2 \in Y_2$. Otherwise, we enlarge the dimension of $N_{j_1 j_2}$ by inserting ``pseudo vertices" so that $N_{j_1 j_2}$ is a square matrix.

We say that $N$ satisfies the \textit{Markov condition of order $k$} if there exists a nontrivial subspace $\{V_J\}_{J \in Y_{k+1}}$ such that, for each $J \in Y_{k+1}$, there exists $m_{J(0, k-1), J(1, k)}$ such that $V_{J(0, k-1)} N_J = m_{J(0, k-1), J(1, k)} V_{J(1, k)}$, here $J(i_1, i_2) = j_{i_1} j_{i_1+1} \cdots j_{i_2}$ with $J = j_1 j_2 \cdots j_{k+1}$. For simplification, we say that $N$ satisfies the Markov condition if $N$ satisfies the Markov condition of order $k$ for some $k \in \mathbb{N}$.

At this point, a further question arises:
\begin{quote}
\it Suppose $N$ satisfies the Markov condition. What kind of Markov measure is $\nu$?
\end{quote}
To answer this question, we may assume $m(J( 0,k-1), J(1,k)) \in \mathbb{R}$ such that
\begin{equation}\label{eq:coeff-when-N-markov-condition}
V_{J(0, k-1) }N_{J(0, k)} = m(J(0, k-1), J(1, k)) V_{J(1, k)}.
\end{equation}
In \cite{BCC-2012}, the authors illustrated what kind of Markov measure $\nu$ is.

\begin{theorem}[See {\cite[Theorem 4]{BCC-2012}}]\label{thm-sofic-measure-THM4}
If $N$ satisfies the Markov condition of order $k$, then $Y$ is a SFT. Furthermore, $\nu$ is the unique maximal measure of $Y$ with transition matrix $M = [m(J,J')]_{J,J' \in Y_k}$.
\end{theorem}

To clarify the construction of $N$ and Theorem \ref{thm-sofic-measure-THM4}, we introduce an example which was initiated by Blackwell.

\begin{example}[Blackwell {\cite{Bla-1957}}]
Let $\mathcal{A}(X)=\{1,2,3\}, \mathcal{A}(Y)=\{1,2\}$ and the one-block map $\Phi: \mathcal{A}(X) \rightarrow \mathcal{A}(Y)$ be defined by $\Phi(1)=1,$ $\Phi(2) = \Phi(3) =2$. Let $\phi: X \rightarrow Y$ be the factor induced from $\Phi$, and the transition matrix of $X$ be
\begin{equation*}
A=\left(
\begin{array}{ccc}
0 & 1 & 1 \\
1 & 1 & 0 \\
1 & 0 & 1%
\end{array}%
\right) \mbox{.}
\end{equation*}%
This factor has been proven ({\cite[Example 2.7]{BP-2011}}) to be Markovian. Here we use Theorem \ref{thm-sofic-measure-THM4} to give a criterion for this property. Since
$E_{1}=\left\{ 1\right\} $ and $E_{2}=\left\{ 2,3\right\}$, we see that $\mathbf{m}=e_{2}=2$, and an extra pseudo vertex is needed for $E_1$. For such reason we introduce the new symbols and the corresponding sets $\widehat{E}_{1}$ and $\widehat{E}_{2}$ are as
follows.
\begin{eqnarray*}
D &=&\left\{ 1,2\right\} \times \left\{ 1,2\right\} =\left\{ \left(
1,1\right) ,\left( 1,2\right) ,\left( 2,1\right) ,\left( 2,2\right)
\right\}
, \\
\mbox{ }\widehat{E}_{1} &=&\left\{ 1=\left( 1,1\right) ,\left(
2,1\right) \right\} \mbox{, }\widehat{E}_{2}=\left\{ 2=\left(
1,2\right) ,3=\left( 2,2\right) \right\} .
\end{eqnarray*}%
Therefore,
\begin{equation*}
B = \left(
\begin{array}{cc}
N_{11} & N_{12} \\
N_{21} & N_{22}%
\end{array}%
\right) =\left(
\begin{array}{cccc}
0 & 0 & 1 & 1 \\
0 & 0 & 0 & 0 \\
1 & 0 & 1 & 0 \\
1 & 0 & 0 & 1%
\end{array}%
\right) .
\end{equation*}%
\begin{equation*}
N_{11}=\left(
\begin{array}{cc}
0 & 0 \\
0 & 0%
\end{array}%
\right) \mbox{, }N_{12}=\left(
\begin{array}{cc}
1 & 1 \\
0 & 0%
\end{array}%
\right) \mbox{, }N_{21}=\left(
\begin{array}{cc}
1 & 0 \\
1 & 0%
\end{array}%
\right) \mbox{, }N_{22}=\left(
\begin{array}{cc}
1 & 0 \\
0 & 1%
\end{array}%
\right) \mbox{.}
\end{equation*}%
Taking $V_{1}=(1\ 0) $ and $V_{2}=(1\ 1) $, one can easily check that $N=(N_{ij})_{i,j=1}^{2}$ satisfies the Markov condition of order $1$. Thus Theorem \ref{thm-sofic-measure-THM4} is applied to show that the factor is a Markov map.
\end{example}

\subsection{Proof of Theorem \ref{main-thm-FSE}}

Proposition \ref{prop-factor-like-embed} asserts that the existence of a factor-like matrix for $T^{(1)}, T^{(2)}$ together with the topological conjugacy of $\phi^{(i)}$ infers there is a map $\overline{\pi}: \mathbf{Y}^{(i)} \to \mathbf{Y}^{(\overline{i})}$ that preserves topological entropy, where $i = 1, 2$, and $i + \overline{i} = 3$. A natural question is whether or not we can find a map connecting $\mathbf{Y}^{(1)}$ and $\mathbf{Y}^{(2)}$ under the condition neither $\phi^{(1)}$ nor $\phi^{(2)}$ is topological conjugacy. The answer is affirmative. First we define the product of scalar and alphabet.

\begin{definition}
Suppose $\mathcal{A}$ is an alphabet set. Let $\mathbf{A}$ be the free abelian additive group generated by $\mathcal{A} \cup \{\varnothing\}$, here $\varnothing$ is the identity element. For $\mathbf{k} \in \mathbb{Z}, \mathbf{a} \in \mathcal{A} \cup \{\varnothing\}$, we define an commutative operator $*$ by
$$
\mathbf{a} * \mathbf{k} = \mathbf{k} * \mathbf{a} = \left\{
           \begin{array}{ll}
             \mathbf{k}\mathbf{a}, & \hbox{if $\mathbf{a} \neq \varnothing$ and $\mathbf{k} \neq 0$;} \\
             \varnothing, & \hbox{otherwise.}
           \end{array}
         \right.
$$
\end{definition}

Suppose $S$ is an $m \times n$ symbolic matrix and $A$ is an $n \times k$ integral matrix. The product $S*A$ is defined by $(S*A)(p, q) = \sum_{i=1}^n S(p, i) * A(i, q)$ for $1 \leq p \leq m, 1 \leq q \leq k$. For simplicity we denote $S*A$ by $SA$. Similarly, we can define $A*S$ and denote by $AS$ for $m \times n$ integral matrix $A$ and $n \times k$ symbolic matrix $S$.

The following proposition, which is an extension of Proposition \ref{prop-factor-like-embed}, can be verified with a little modification of the proof of Proposition 3.15 in \cite{BCL-JDE2012}. Hence we omit the detail.

\begin{proposition}\label{prop-factor-like-Y-W}
Let $S^{(i)}$ be the symbolic transition matrix of $\mathbf{Y}^{(i)}$ for $i = 1, 2$. Suppose $E$ is a factor-like matrix such that $S^{(i)}E = E S^{(\overline{i})}$, then there exist maps $\pi: W^{(i)} \to W^{(\overline{i})}$ and $\overline{\pi}: \mathbf{Y}^{(i)} \to \mathbf{Y}^{(\overline{i})}$  that both preserve topological entropy, where $i + \overline{i} = 3$.
\end{proposition}

A factor map $\phi$ is \emph{almost invertible} if every doubly transitive point has exactly one preimage. Lemma \ref{lem-ai-sync-word} shows that the existence of a synchronizing word is a necessary and sufficient criterion whether $\phi$ is almost invertible.

\begin{lemma}\label{lem-ai-sync-word}
Suppose $\phi : X \to Y$ is a one-block factor map. Then $\phi$ is almost invertible if and only if $\phi$ has a synchronizing word.
\end{lemma}
\begin{proof}
If $\phi$ is almost invertible, then $d_{\phi}^* = 1$. Let $w$ be a magic word and $i$ be a magic coordinate. In other words, $d_{\phi}(w, i) = 1$. The fact that $\phi$ is right-resolving infers that $d_{\phi}(w, |w|) = 1$. Hence $w$ is a synchronizing word.

On the other hand, suppose $w$ is a synchronizing word. $\phi$ is right-resolving indicates $d_{\phi}(wa, |wa|) = 1$ for some $a$ such that $wa \in B(Y)$. That is, $wa$ is a magic word and $d_{\phi}^* = 1$. Therefore, $\phi$ is almost invertible.
\end{proof}

The proof of the first statement of Theorem \ref{main-thm-FSE} is done by Lemma \ref{lem-ai-sync-word} and the following theorem.

\begin{theorem}[See {\cite[Theorem 3.4]{MS-PJM2001}}]\label{thm-MS-max2max}
Suppose $\phi: X \to Y$ is a factor map and $X$ is an irreducible SFT. If $\phi$ is almost invertible, then $\phi: \mathcal{M}_{\max}(X) \to \mathcal{M}_{\max}(Y)$ is a bijection. Moreover, $h_{\mu}(X) = h_{\phi \mu}(Y)$ for $\mu \in \mathcal{M}_{\max}(X)$.
\end{theorem}

Next we continue the proof of Theorem \ref{main-thm-FSE}.

Fix $i \in \{1, 2\}$. Recall that the metric $d^{(i)}: W^{(i)} \times W^{(i)} \to \mathbb{R}$ is given by
$$
d^{(i)}(x, y) = \sum_{j \in \mathbb{Z}} \frac{|x_j - y_j|}{m_i^{|j|+1}},
$$
for $x, y \in W^{(i)}$, where $m_i = |\mathcal{A}(W^{(i)})|$.

To formulate the explicit form of the Hausdorff dimension of the hidden and output spaces, we introduce the following from Pesin's well-known work.

\begin{theorem}[See {\cite[Theorems 13.1 and 22.2]{P-1997}}]\label{thm-pesin-dim}
Let $(X, \sigma)$ be a shift space with $|\mathcal{A}(X)| = m$, and $0 < \lambda_1, \ldots, \lambda_m < 1$. Suppose $d$ is a metric defined on $X$. If there exist $K_1, K_2 > 0$ such that
\begin{align*}
K_1 \prod_{j=0}^{n_2} \lambda_{i_j} < &\mathrm{diam} [i_0, \ldots, i_{n_2}] < K_2 \prod_{j=0}^{n_2} \lambda_{i_j}, \\
K_1 \prod_{j=0}^{n_1} \lambda_{i_j} < &\mathrm{diam} [i_{-n_1}, \ldots, i_0] < K_2 \prod_{j=0}^{n_1} \lambda_{i_j},
\end{align*}
for any cylinder $I = [i_{-n_1}, \ldots, i_{n_2}], n_1, n_2 \geq 0$, then
$$
\dim X = - \frac{h_{\mu_{\lambda}}(X)}{\int_X \log \lambda_{i_0} d \mu_{\lambda}} = - 2 \frac{h_{\mu^{\pm}_{\lambda}}(X)}{\int_X \log \lambda_{i_0} d \mu^{\pm}_{\lambda}},
$$
where $\mu_{\lambda}$ is a maximal measure on $X$ and $\mu^{\pm}_{\lambda}$ is a maximal measure on the right-sided subspace $X^+$/left-sided subspace $X^-$.
\end{theorem}

Suppose $\mu^{(i)}$ is a maximal measure of $W^{(i)}$. For any cylinder $I=[i_{-n_1}, \ldots, i_{n_2}]$, the diameter of $[i_0, \ldots, i_{n_2}]$ and $[i_{-n_1}, \ldots, i_0]$ are $\dfrac{1}{m_i^{n_2+1}}$ and $\dfrac{1}{m_i^{n_1+1}}$ respectively. Let $K_1=1, K_2=3$, and $\lambda_1 = \cdots = \lambda_{m_i} = \dfrac{1}{m_i}$, apply Theorem \ref{thm-pesin-dim} we have
$$
\dim W^{(i)} = - \frac{h_{\mu^{(i)}}(W^{(i)})}{\displaystyle\int_{W^{(i)}} \log \dfrac{1}{m_i} d \mu^{(i)}} = \frac{h_{\mu^{(i)}}(W^{(i)})}{\log m_i}  = 2 \frac{h_{\mu^{(i),\pm}}(W^{(i)})}{\log m_i}.
$$
Moreover, the one-to-one correspondence between $\mathcal{M}_{\max}(W^{(i)})$ and $\mathcal{M}_{\max}(\mathbf{Y}^{(i)})$ demonstrates that
\begin{align*}
\dim \mathbf{Y}^{(i)} &= \sup \{\frac{h_{\nu}(\mathbf{Y}^{(i)})}{\log n_i}: \nu \text{ is invariant on } \mathbf{Y}^{(i)}\} \\
 &= \frac{h_{\phi^{(i)}\mu^{(i)}}(\mathbf{Y}^{(i)})}{\log n_i} = \frac{h_{\mu^{(i)}}(W^{(i)})}{\log n_i}.
\end{align*}
The last equality comes from Theorem \ref{thm-MS-max2max}. This completes the proof of Theorem \ref{main-thm-FSE} part (ii).

Observe that
$$
\dim \mathbf{Y}^{(i)} = \frac{h_{\nu^{(i)}}(\mathbf{Y}^{(i)})}{\log n_i}
$$
indicates $\nu^{(i)}$ is a maximal measure on $\mathbf{Y}^{(i)}$. Since $W^{(i)}$ is irreducible, the maximal measure $\mu^{(i)}$ is unique. Hence there is a bijection $\overline{\pi}: \mathcal{M}_{\max}(\mathbf{Y}^{(i)}) \to \mathcal{M}_{\max}(\mathbf{Y}^{(\overline{i})})$ such that $\overline{\pi} \nu^{(i)} = \nu^{(\overline{i})}$. Therefore,
$$
\dim \mathbf{Y}^{(\overline{i})} = \frac{h_{\overline{\pi}\nu^{(i)}}(\mathbf{Y}^{(\overline{i})})}{\log n_{\overline{i}}} = \frac{h_{\nu^{(\overline{i})}}(\mathbf{Y}^{(\overline{i})})}{\log n_{\overline{i}}}.
$$
This completes the proof of Theorem \ref{main-thm-FSE}.

\subsection{Proof of Theorem \ref{main-thm-ITO}}

Whether there exists a factor map connecting two spaces is always a concerning issue. In general, it is difficult to construct or to say such factor maps exist for a given pair of spaces. Proposition \ref{prop-factor-like-Y-W} proposes a methodology for constructing a connection between two spaces. Notably a map constructed via Proposition \ref{prop-factor-like-Y-W} preserves topological entropy. In other words, it only works for those spaces reaching the same topological entropy if we restrict the factor maps. In this subsection, we turn our attention to the factor maps connecting spaces with non-equal topological entropies.

Similar to the proof of Theorem \ref{main-thm-FSE}, demonstrating Theorem \ref{main-thm-ITO} relies mainly on the existence of a factor map. Instead of $\mathbf{Y}^{(1)}, \mathbf{Y}^{(2)}$, we start with examining whether there is a factor map from $W^{(i)}$ to $W^{(\overline{i})}$; note here that $i + \overline{i} = 3$.

\begin{theorem}\label{thm-inf-to-one-on-W}
Suppose $W^{(1)}$ and $W^{(2)}$ are irreducible with $h(W^{(1)}) \neq h(W^{(2)})$. Suppose $h(W^{(i)}) > h(W^{(\overline{i})})$, where $i + \overline{i} = 3$. Then there exists an infinite-to-one map $\pi: W^{(i)} \to W^{(\overline{i})}$ if one of the following is satisfied.
\begin{enumerate}[\bf a)]
\item $h(W^{(i)}) = h(\mathbf{Y})$ and there is a factor-like matrix $F$ such that $T^{(i)} F = F T$, where $T$ is the transition matrix of $\mathbf{Y}$.

\item $h(W^{(i)}) < h(\mathbf{Y})$.
\end{enumerate}
\end{theorem}

\begin{remark}\label{rmk-for-thm-inf-to-one-on-W}
\begin{enumerate}[\bf (i)]
\item Suppose $X, Y$ are two irreducible SFTs with $h(X) > h(Y)$. In \cite{Kit-1998}, Kitchens showed that if there is an infinite-to-one factor map from $X^+$ to $Y^+$, then there exists an infinite-to-one factor map $\pi: X \to Y$. This reduces the investigation of Theorem \ref{thm-inf-to-one-on-W} to the existence of an infinite-to-one map between the right-sided subspaces of $W^{(1)}$ and $W^{(2)}$.

\item Theorem \ref{thm-inf-to-one-on-W} reveals the existence of an infinite-to-one map between the hidden and output spaces whenever these two spaces hit different topological entropies; however, there are an infinite number of such maps general. In addition, it is difficult to find the explicit form of an infinite-to-one map. This is an important issue and is still open in the field of symbolic dynamical systems. It helps for the investigation of MCNNs if one can propose a methodology to find a concrete expression of an infinite-to-one map.
\end{enumerate}
\end{remark}

The following corollary comes immediately after Theorem \ref{thm-inf-to-one-on-W}.

\begin{corollary}
Under the same assumption of Theorem \ref{thm-inf-to-one-on-W}. Suppose furthermore that $|\mathcal{A}(W^{(i)})| \geq |\mathcal{A}(\mathbf{Y})|$ and \textbf{a)} is satisfied. Then $\pi: W^{(i)} \to W^{(\overline{i})}$ is an infinite-to-one factor map.
\end{corollary}

Suppose $X$ is a shift space. Let $P(X)$ denote the collection of periodic points in $X$ and let $P_n(X)$ be the set of periodic points with period $n$. Given two shifts $X$ and $Y$, let $q_n(X)$ and $q_n(Y)$ be the cardinality of $\cup_{k \geq n} P_k(X)$ and $\cup_{k \geq n} P_k(Y)$, respectively. If $q_n(X) \leq q_n(Y)$ for $n \geq 1$, then we call it an \emph{embedding periodic point condition}, and write it as $P(X) \hookrightarrow P(Y)$. Embedding Theorem asserts a necessary and sufficient condition whether there exists an injective map between $X$ and $Y$.

\begin{theorem}[Embedding Theorem]
Suppose $X$ and $Y$ are irreducible SFTs. There is an embedding map $\phi: X \to Y$ if and only if $h(X) < h(Y)$ and $P(X) \hookrightarrow P(Y)$.
\end{theorem}

A forthcoming question is the existence of a factor map between $X$ and $Y$. Like the embedding periodic point condition, the \emph{factor periodic point condition} indicates that, for every $x \in P_n(X)$, there exists a $y \in P_m(Y)$ such that $m$ is a factor of $n$, and is denoted by $P(X) \searrow P(Y)$.

\begin{theorem}[See {\cite[Theorem 4.4.5]{Kit-1998}}]\label{thm-inf-to-1-SFT}
Suppose $X$ and $Y$ are irreducible SFTs. There exists an infinite-to-one factor code $\phi: X \to Y$ if and only if $h(X) > h(Y)$ and $P(X) \searrow P(Y)$.
\end{theorem}

\begin{proof}[Proof of Theorem \ref{thm-inf-to-one-on-W}]
Without the loss of generality, we may assume that $h(W^{(1)}) < h(W^{(2)})$. It suffices to demonstrate there is an infinite-to-one map from $W^{(2), +}$ to $W^{(1), +}$ due to the observation in Remark \ref{rmk-for-thm-inf-to-one-on-W} (i). For the ease of notation, the spaces in the upcoming proof are referred to as right-sided subspaces.

Suppose that condition \textbf{a)} is satisfied. The existence of factor-like matrix $F$ such that $T^{(2)} F = F T$ implies there is a map $\Phi^{(2)}: W^{(2)} \to \mathbf{Y}$.

Recall that the graph representation $G^{(1)}$ of $W^{(1)}$ is obtained by applying subset construction to $(G, \mathcal{L}^{(1)})$. Without the loss of generality, we assume that $G^{(1)}$ is essential. That is, every vertex in $G^{(1)}$ is treated as an initial state of one edge and as a terminal state of another. Suppose $w = w_1 \cdots w_n$ is a cycle in $G$. If the initial state $i(w_k)$ of $w_k$ is a vertex in $G^{(1)}$ for $k = 1, \ldots, n$, then $w$ is also a cycle in $G^{(1)}$.

Assume $k$ is the only index that either $i(w_k)$ or $t(w_k)$ is not a vertex in $G^{(1)}$, where $t(e)$ denotes the terminal vertex of the edge $e$. First we consider that only one of these two vertices is not in $G^{(1)}$. For the case that $i(w_k)$ is not a vertex in $G^{(1)}$, there is a vertex, say $v_k$, in $G^{(1)}$ so that $v_k$ is a grouping vertex which contains $i(w_k)$.\footnote{If fact, each vertex in $G^{(1)}$ is the grouping of one or more vertices in $G$, and so is $G^{(2)}$. The reader is referred to \cite{BCLL-JDE2009} for more details.} Hence there is an edge $\overline{w}_{k-1}$ in $G^{(1)}$ such that $i(\overline{w}_{k-1}) = i(w_{k-1})$ and $t(\overline{w}_{k-1}) = v_k$. In other words, there is an edge in $G^{(1)}$ that can be related to $w_{k-1}$. Moreover, there is an edge $(v_k, t(w_k))$ in $G^{(1)}$ if $t(w_k)$ is a vertex in $G^{(1)}$. Hence there is a cycle in $G^{(1)}$ that corresponds to $w$. The case that $t(w_k)$ is not a vertex in $G^{(1)}$ can be conducted in an analogous discussion. For the case that both the initial and terminal states of $w_k$ are not in $G^{(1)}$, combining the above demonstration infers there is a new vertex $v_{k+1}$ and two new edges $e_k=(v_k, v_{k+1}), e_{k+1} = (v_{k+1}, t(w_{k+1}))$ in $G^{(1)}$. That is, there is still a cycle in $G^{(1)}$ that corresponds to $w$.

Repeating the above process if necessary, it is seen that, for every cyclic path in $G$ with length $n$, there is an associated cyclic path in $G^{(1)}$ with length $m$ and $m$ divides $n$. Theorem \ref{thm-inf-to-1-SFT} asserts there exists an infinite-to-one factor $\Phi^{(1)}: \mathbf{Y} \to W^{(1)}$. Let $\pi = \Phi^{(1)} \circ \Phi^{(2)}$. Then $\pi$ is an infinite-to-one map from $W^{(2)} \to W^{(1)}$ by Theorem \ref{thm-diamond-inf2one}.

Next, for another case, suppose that condition \textbf{b)} is satisfied. It suffices to demonstrate the existence of an embedding map from $W^{(2)}$ to $\mathbf{Y}$. The elucidation of the existence of a map from $W^{(2)}$ to $\mathbf{Y}$ can be performed via a similar but converse argument as with the discussion of $\Phi^{(1)}$. Hence we omit the details. Since the graph representation $G^{(2)}$ of $W^{(2)}$ comes from applying subset construction to $(G, \mathcal{L}^{(2)})$, it can be verified that every periodic point in $W^{(2)}$ corresponds to a cyclic path in $G^{(2)}$, and, for every cyclic path in $G^{(2)}$, we can illustrate a cyclic path in $G$. The Embedding Theorem demonstrates the existence of an embedding map $\overline{\Phi}^{(2)}: W^{(2)} \to \mathbf{Y}$.

This completes the proof.
\end{proof}

Once we demonstrate the existence of a factor map $\pi: W^{(i)} \to W^{(\overline{i})}$, the proof of Theorem \ref{main-thm-ITO} can be performed via analogous method as the proof of Theorem \ref{main-thm-FSE}. Hence we skip the proof. Instead, it is interesting if there is a criterion to determine whether $\pi$ is uniform.

\begin{theorem}\label{thm-uniform-iff-W}
Suppose $N$, obtained from $\pi$ as defined in the previous subsection, satisfies the Markov condition of order $k$. Define
$$
M = [m_{J(0, k-1), J(1, k)}]_{J \in W^{(1)}_{k+1}},
$$
here $m_{J(0, k-1), J(1, k)}$ is defined by \eqref{eq:coeff-when-N-markov-condition}. Then $\pi$ is uniform if and only if
\begin{equation}\label{eq-uniform-cond-on-W}
\rho_M = \dfrac{\rho^{(2)}}{\rho^{(1)}}
\end{equation}
where $\rho^{(i)}$ is the spatial radius of the transition matrix $T^{(i)}$ of $W^{(i)}$ for $i = 1, 2$.
\end{theorem}

Theorem \ref{thm-uniform-iff-W} is obtained with a little modification of the proof of Proposition 6.1 in \cite{BT-TAMS1984}, thus we omit it here. The following corollary comes immediately from Theorem \ref{thm-uniform-iff-W}.

\begin{corollary}\label{cor-dim-inf2one}
Let $N$ be defined as above. Suppose $N$ satisfies the Markov condition and \eqref{eq-uniform-cond-on-W} holds. Then $\overline{\pi} \nu^{(i)} \in \mathcal{M}_{\max}(\mathbf{Y}^{(\overline{i})})$ if $\nu^{(i)} \in \mathcal{M}_{\max}(\mathbf{Y}^{(i)})$. Furthermore, if
$$
\dim \mathbf{Y}^{(i)} = \dfrac{h_{\nu^{(i)}}(\mathbf{Y}^{(i)})}{\log n_i},
$$
then
$$
\dim \mathbf{Y}^{(\overline{i})} = \dfrac{h_{\overline{\pi} \nu^{(i)}}(\mathbf{Y}^{(\overline{i})})}{\log n_{\overline{i}}}.
$$
\end{corollary}

\section{Examples}

\begin{example}\label{eg-Y1Y2-SFT}
Suppose the templates of a SMCNN are given by the following:
\begin{align*}
[a^{(1)}, a_r^{(1)}, z^{(1)}] &= [2.9, 1.7, 0.1] \\
[a^{(2)}, a_r^{(2)}, b^{(2)}, b_r^{(2)}, z^{(2)}] &= [-0.3, -1.2, 0.7, 2.3, 0.9]
\end{align*}
Then the basic set of admissible local patterns is
$$
\mathcal{B} = \left\{
\boxed{-+ \atop \displaystyle --}\,, \boxed{-+ \atop \displaystyle +-}\,, \boxed{+- \atop \displaystyle -+}\,, \boxed{+- \atop \displaystyle ++}\,, \boxed{++ \atop \displaystyle -+}\,, \boxed{++ \atop \displaystyle ++}\right\}.
$$
The transition matrix $T$ of the solution space $\mathbf{Y}$ is
$$
T = \begin{pmatrix}
        0 & 0 & 1 & 0 \\
        0 & 0 & 1 & 0 \\
        0 & 1 & 0 & 1 \\
        0 & 1 & 0 & 1 \\
      \end{pmatrix},
$$
and the symbolic transition matrices of the hidden and output spaces are
$$
S^{(1)} = \begin{pmatrix}
        \varnothing & \alpha_1 \\
        \alpha_2 & \alpha_3 \\
      \end{pmatrix}
\quad and \quad
S^{(2)} = \begin{pmatrix}
        \varnothing & \alpha_1 & \varnothing \\
        \alpha_2 & \varnothing & \alpha_3 \\
        \alpha_2 & \varnothing & \alpha_3 \\
      \end{pmatrix}
$$
respectively. Figure \ref{fig-551429} shows that $\mathbf{Y}^{(1)}$ and $\mathbf{Y}^{(2)}$ are two different spaces. The topological entropy of $\mathbf{Y}^{(i)}$ is related to the spectral radius of the incidence of $S^{(i)}$. An easy computation infers $h(\mathbf{Y}^{(1)}) = h(\mathbf{Y}^{(2)}) = \log g$, where $g = (1+\sqrt{5})/2$ is the golden mean.

Let
$$
E = \begin{pmatrix}
        1 & 0 \\
        0 & 1 \\
        0 & 1 \\
      \end{pmatrix}.
$$
Then $S^{(2)} E = E S^{(1)}$. Proposition \ref{prop-factor-like-Y-W} indicates that there exist factor maps $\pi: W^{(2)} \to W^{(1)}$ and $\overline{\pi}: \mathbf{Y}^{(2)} \to \mathbf{Y}^{(1)}$. More precisely, let
\begin{align*}
&\mathcal{A}(W^{(1)}) = \{x_1, x_2\}, \quad \mathcal{A}(W^{(2)}) = \{x'_1, x'_2, x'_3\}; \\
&\mathcal{A}(\mathbf{Y}^{(1)}) = \{y_1, y_2\}, \quad \mathcal{A}(\mathbf{Y}^{(2)}) = \{y'_1, y'_2, y'_3\}.
\end{align*}
Then
$$
\pi(x'_1) = x_1, \quad \pi(x'_2) = \pi(x'_3) = x_2, \quad \overline{\pi}(y'_1) = y_1, \quad \overline{\pi}(y'_2) = \overline{\pi}(y'_3) = y_2.
$$
See Figure \ref{fig-eg-Y1Y2-SFT}.

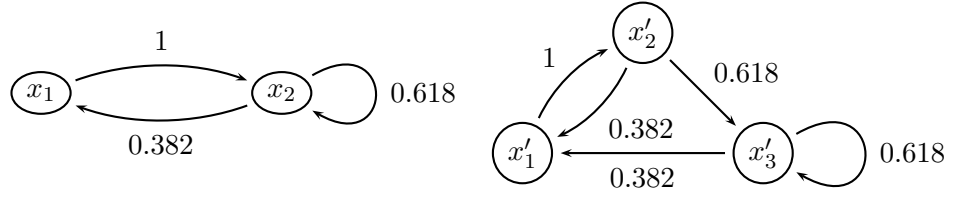
\begin{figure}
\begin{center}
\psset{unit=0.8cm}
\begin{pspicture}(13,3)
\psset{nodesep=0.1cm}
\rput(1,2){\ovalnode{A}{$x_1$}}  \rput(5,2){\ovalnode{B}{$x_2$}}
\rput(9,1){\ovalnode{C}{$x_1'$}}  \rput(13,1){\ovalnode{D}{$x_3'$}}  \rput(11,3){\ovalnode{E}{$x_2'$}}

\ncarc[arcangle=20]{->}{A}{B}\Aput{$1$}  \ncarc[arcangle=20]{->}{B}{A}\Aput{$0.382$}  \nccurve[angleA=30,angleB=-30,ncurv=5.5]{->}{B}{B}\Aput{$0.618$}
\ncarc[arcangle=20]{->}{C}{E}\Aput{$1$}  \ncarc[arcangle=20]{->}{E}{C}\Aput{$0.382$}
\ncline{->}{E}{D}\Aput{$0.618$}
\nccurve[angleA=30,angleB=-30,ncurv=5.5]{->}{D}{D}\Aput{$0.618$}
\ncline{->}{D}{C}\Aput{$0.382$}
\end{pspicture}
\end{center}
\caption{The graph representation of the hidden and output spaces of Example \ref{eg-Y1Y2-SFT}. The number on the edge is the transition probability. The left one represents $\mathbf{Y}^{(1)}$ and the right one represents $\mathbf{Y}^{(2)}$.} \label{fig-eg-Y1Y2-SFT}
\end{figure}

Suppose $\mathcal{A}(\mathbf{Y}) = \{z_1, z_2, z_3, z_4\}$, the factor map $\psi: \mathbf{Y} \to \mathbf{Y}^{(1)}$ is given by
$$
\psi(z_1) = \psi(z_3) = y_1, \quad \psi(z_2) = \psi(z_4) = y_2
$$
Set $N = (N_{ij})_{i \leq i, j \leq 2}$ and $L_1, L_2$ by
$$
N_{11} = N_{21} = \begin{pmatrix}
        0 & 1 \\
        0 & 0 \\
      \end{pmatrix},
\quad
N_{12} = N_{22} = \begin{pmatrix}
        0 & 0 \\
        1 & 1 \\
      \end{pmatrix}
$$
and $L_1 = (0\ 1), L_2 = (g\ g)$, respectively. A straightforward calculation demonstrates that
$$
L_1 N_{11} = 0 \cdot L_1, L_1 N_{12} = g^{-1} \cdot L_2, L_2 N_{11} = g \cdot L_1, L_2 N_{22} = 1 \cdot L_2.
$$
That is, $N$ satisfies the Markov condition of order $1$. Theorem \ref{thm-sofic-measure-THM4} indicates that $\mathbf{Y}^{(1)}$ is a SFT with the unique maximal measure of entropy $\nu^{(1)}$, and $\nu^{(1),+} = (p_{\mathbf{Y}^{(1)}}, P_{\mathbf{Y}^{(1)}})$, where $p_{\mathbf{Y}^{(1)}} = (\dfrac{2-g}{3-g}, \dfrac{1}{3-g})$ and
$$
P_{\mathbf{Y}^{(1)}} = \begin{pmatrix}
        0 & 1 \\
        2-g & g-1 \\
      \end{pmatrix}
    = stoch(M),
\quad
M = \begin{pmatrix}
        0 & 1/g \\
        g & 1 \\
      \end{pmatrix}.
$$
Applying Theorem \ref{main-thm-FSE}, we have
\begin{align*}
\dim W^{(1)} &= \dim \mathbf{Y}^{(1)}= 2 \frac{h_{\nu^{(1),+}}(\mathbf{Y}^{(1)})}{\log 2} \\
 &= \frac{2}{(g-3)\log 2} ((2-g) \log (2-g) + (g-1) \log (g-1)) = 2 \frac{\log g}{\log 2}.
\end{align*}

On the other hand,
$$
S^{(2)}S^{(2)} = \begin{pmatrix}
        \alpha_1 \alpha_2 & \varnothing & \alpha_1 \alpha_3 \\
        \alpha_3 \alpha_2 & \alpha_2 \alpha_1 & \alpha_3 \alpha_3 \\
        \alpha_3 \alpha_2 & \alpha_2 \alpha_1 & \alpha_3 \alpha_3 \\
      \end{pmatrix}
$$
infers that every word of length $3$ in $\mathbf{Y}^{(2)}$ is a synchronizing word. That is, $\mathbf{Y}^{(2)}$ is topological conjugate to $W^{(2)}$. Since the unique maximal measure of $W^{(2)}$ is $\mu^{(2)}$ with $\mu^{(2),+} = (p_{W^{(2)}}, P_{W^{(2)}})$, where $p_{W^{(2)}} = (\dfrac{2-g}{3-g}, \dfrac{2-g}{3-g}, \dfrac{g-1}{3-g})$,
$$
P_{W^{(2)}} = \begin{pmatrix}
        0 & 1 & 0 \\
        2-g & 0 & g-1 \\
        2-g & 0 & g-1 \\
      \end{pmatrix}.
$$
Theorem \ref{main-thm-FSE} suggests that
\begin{align*}
\dim W^{(2)} &= 2 \frac{h_{\mu^{(2),+}}(W^{(2)})}{\log 3} = 2 \frac{\log g}{\log 3}, \\
\intertext{and}
\dim \mathbf{Y}^{(2)} &= 2 \frac{h_{\phi^{(2)}\mu^{(2),+}}(\mathbf{Y}^{(2)})}{\log 2} = 2 \frac{h_{\mu^{(2),+}}(W^{(2)})}{\log 2} = 2 \frac{\log g}{\log 2}.
\end{align*}
$\pi\mu^{(2)} = \mu^{(1)}$ can be verified without difficulty, thus we omit the details.
\end{example}

\begin{figure}
\begin{center}
\includegraphics[scale=0.7]{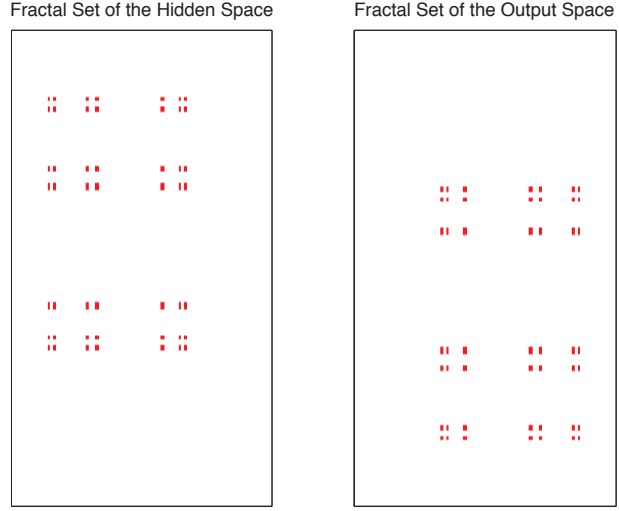}
\caption{The fractal sets of  the hidden and output spaces of Example \ref{eg-FSE-AI}. The templates are given by $[a^{(1)}, a_r^{(1)}, z^{(1)}] = [2.9, 1.7, 0.1]$ and $[a^{(2)}, a_r^{(2)}, b^{(2)}, b_r^{(2)}, z^{(2)}] = [-0.1, -1.1, 2.1, -1.4, 0.9]$. The output space $\mathbf{Y}^{(2)}$ is a strict sofic shift rather than a SFT. Meanwhile, the hidden space $\mathbf{Y}^{(1)}$ is a SFT.}
\label{fig-642428}
\end{center}
\end{figure}

\begin{example}\label{eg-FSE-AI}
Suppose the template of the first layer is the same as in Example \ref{eg-Y1Y2-SFT}, and
$$
[a^{(2)}, a_r^{(2)}, b^{(2)}, b_r^{(2)}, z^{(2)}] = [-0.1, -1.1, 2.1, -1.4, 0.9].
$$
The basic set of admissible local patterns of the solution space $\mathbf{Y}$ is
$$
\mathcal{B} = \left\{
\boxed{-+ \atop \displaystyle -+}\,, \boxed{-- \atop \displaystyle -+}\,, \boxed{+- \atop \displaystyle +-}\,, \boxed{++ \atop \displaystyle +-}\,, \boxed{+- \atop \displaystyle ++}\,, \boxed{+- \atop \displaystyle --}\right\}.
$$
The transition matrix $T$ of the solution space $\mathbf{Y}$ is
$$
T = \begin{pmatrix}
        0 & 1 & 0 & 1 \\
        0 & 0 & 0 & 0 \\
        1 & 0 & 0 & 0 \\
        1 & 1 & 1 & 0 \\
      \end{pmatrix}.
$$
After careful examination, the hidden and output spaces are both mixing with symbolic transition matrices
$$
S^{(1)} = \begin{pmatrix}
            \varnothing & \varnothing & \alpha_1 \\
            \alpha_0 & \varnothing & \alpha_1 \\
            \varnothing & \alpha_2 & \varnothing \\
          \end{pmatrix}, \quad
S^{(2)} = \begin{pmatrix}
            \varnothing & \varnothing & \alpha_1 & \varnothing \\
            \alpha_2 & \varnothing & \varnothing & \varnothing \\
            \varnothing & \alpha_3 & \varnothing & \alpha_2 \\
            \varnothing & \varnothing & \alpha_1 & \varnothing \\
          \end{pmatrix}.
$$
See Figure \ref{fig-642428}. $\mathbf{Y}^{(1)}$ and $\mathbf{Y}^{(2)}$ are FSE since $h(\mathbf{Y}^{(1)}) = h(\mathbf{Y}^{(2)}) = \log \rho$, where $\rho \approx 1.3247$ satisfies $\rho^3 - \rho - 1 = 0$. Let
$$
E = \begin{pmatrix}
      0 & 0 & 1 \\
      1 & 0 & 0 \\
      0 & 1 & 0 \\
      0 & 0 & 1 \\
    \end{pmatrix}.
$$
Notably, $T^{(2)} E = E T^{(1)}$ and there exists no factor-like matrix $F$ such that $S^{(2)} F = F S^{(1)}$ or $S^{(1)} F = F S^{(2)}$.
It follows from $S^{(1)}$ that every word of length $2$ in $\mathbf{Y}^{(1)}$ is a synchronizing word. Hence $\mathbf{Y}^{(1)} \cong W^{(1)}$. The unique maximal measure of entropy for $W^{(1),+}$ is $\mu^{(1),+} = (p_{W^{(1)}}, P_{W^{(1)}})$, where $p_{W^{(1)}} = (0.1770, 0.4115, 0.4115)$ and
$$
P_{W^{(1)}} = \begin{pmatrix}
        0 & 0 & 1 \\
        0.4302 & 0 & 0.5698 \\
        0 & 1 & 0 \\
      \end{pmatrix}.
$$
Hence
\begin{align*}
\dim W^{(1)} &= 2 \frac{h_{\mu^{(1),+}}(W^{(1)})}{\log 3} \approx 0.5119, \\
\intertext{and}
\dim \mathbf{Y}^{(1)} &= 2 \frac{h_{\phi^{(1)}\mu^{(1),+}}(\mathbf{Y}^{(1)})}{\log 2} = 2 \frac{h_{\mu^{(1),+}}(W^{(1)})}{\log 2} \approx 0.8114.
\end{align*}

Unlike Example \ref{eg-Y1Y2-SFT}, it can be checked (with or without computer assistance) that $\mathbf{Y}^{(2)}$, rather than a SFT, is a strict sofic shift since there exists no $k \in \mathbb{N}$ such that every word of length $k$ is a synchronizing word in $\mathbf{Y}^{(2)}$. Nevertheless, there is a synchronizing word of length $2$ (that is, $\alpha_3 = ++$). Theorem \ref{main-thm-FSE} (i) indicates that there is a one-to-one correspondence between $\mathcal{M}_{\max}(W^{(2)})$ and $\mathcal{M}_{\max}(\mathbf{Y}^{(2)})$. Since the unique maximal measure of $W^{(2),+}$ is $\mu^{(2),+} = (p_{W^{(2)}}, P_{W^{(2)}})$, where $p_{W^{(2)}} = (0.1770, 0.1770, 0.4115, 0.2345)$ and
$$
P_{W^{(2)}} = \begin{pmatrix}
        0 & 0 & 1 & 0 \\
        1 & 0 & 0 & 0 \\
        0 & 0.4302 & 0 & 0.5698 \\
        0 & 0 & 1 & 0 \\
      \end{pmatrix},
$$
we have
\begin{align*}
\dim W^{(2)} &= 2 \frac{h_{\mu^{(1),+}}(W^{(1)})}{\log 4} \approx 0.4057, \\
\intertext{and}
\dim \mathbf{Y}^{(2)} &= 2 \frac{h_{\phi^{(2)}\mu^{(2),+}}(\mathbf{Y}^{(2)})}{\log 2} = 2 \frac{h_{\mu^{(2),+}}(W^{(2)})}{\log 2} \approx 0.8114.
\end{align*}
\end{example}

\begin{figure}
\begin{center}
\includegraphics[scale=0.7]{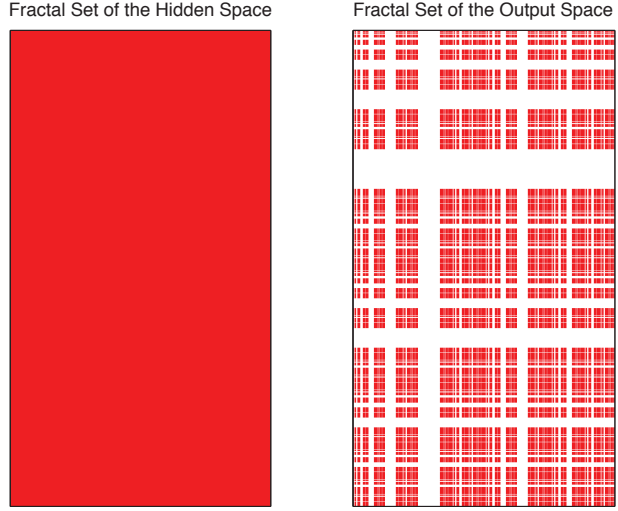}
\caption{The fractal sets of  the hidden and output spaces of Example \ref{eg-ITO-Y1-SFT}. The templates are given by $[a^{(1)}, a_r^{(1)}, z^{(1)}] = [2.9, 1.7, 0.1]$ and $[a^{(2)}, a_r^{(2)}, b^{(2)}, b_r^{(2)}, z^{(2)}] = [1.3, -1.2, 0.7, 2.3, 0.8]$. It is seen that the hidden space $\mathbf{Y}^{(1)}$ is the unit square $[0, 1] \times [0,1 ]$, and so is $W^{(1)}$. Moreover, there are infinite-to-one factor maps $\pi: W^{(1)} \to W^{(2)}$ and $\pi: \mathbf{Y}^{(1)} \to \mathbf{Y}^{(2)}$.}
\label{fig-551539}
\end{center}
\end{figure}

\begin{example}\label{eg-ITO-Y1-SFT}
Suppose the template of the first layer is the same as in Example \ref{eg-Y1Y2-SFT}, and
\begin{align*}
[a^{(2)}, a_r^{(2)}, b^{(2)}, b_r^{(2)}, z^{(2)}] = [1.3, -1.2, 0.7, 2.3, 0.8].
\end{align*}
Then the basic set of admissible local patterns is
$$
\mathcal{B} = \left\{
\boxed{-- \atop \displaystyle --}\,, \boxed{-+ \atop \displaystyle --}\,, \boxed{-+ \atop \displaystyle +-}\,, \boxed{+- \atop \displaystyle -+}\,, \boxed{+- \atop \displaystyle ++}\,, \boxed{++ \atop \displaystyle --}\,, \boxed{++ \atop \displaystyle -+}\,, \boxed{++ \atop \displaystyle ++}\right\}.
$$
The transition matrix $T$ of the solution space $\mathbf{Y}$,
$$
T = \begin{pmatrix}
        1 & 0 & 1 & 0 \\
        0 & 0 & 1 & 0 \\
        0 & 1 & 0 & 1 \\
        1 & 1 & 0 & 1 \\
      \end{pmatrix}
$$
suggests that $\mathbf{Y}$ is mixing. It is not difficult to see that the symbolic transition matrices of the hidden and output spaces are
$$
S^{(1)} = \begin{pmatrix}
        \alpha_0 & \alpha_1 \\
        \alpha_2 & \alpha_3 \\
      \end{pmatrix}
\quad and \quad
S^{(2)} = \begin{pmatrix}
        \alpha_0 & \varnothing & \alpha_1 & \varnothing & \varnothing \\
        \varnothing & \varnothing & \alpha_1 & \varnothing & \varnothing \\
        \varnothing & \alpha_2 & \varnothing & \alpha_3 & \varnothing \\
        \varnothing & \varnothing & \varnothing & \alpha_3 & \alpha_2 \\
        \alpha_0 & \varnothing & \alpha_1 & \varnothing & \varnothing \\
      \end{pmatrix}
$$
respectively. See Figure \ref{fig-551539} for the fractal sets of $\mathbf{Y}^{(1)}$ and $\mathbf{Y}^{(2)}$.

Obviously $\mathbf{Y}^{(1)}$ is a full $2$-shift. It is remarkable that $\phi^{(1)} \mu^{(1)}$ is not a Markov measure. The unique maximal measure for $W^{(1),+}$ (also for $\mathbf{Y}^{(1),+}$) is the uniform Bernoulli measure $\mu^{(1),+} = (1/2, 1/2)$. Therefore,
$$
\dim W^{(1)} = \dim \mathbf{Y}^{(1)} = 2 \frac{h_{\mu^{(1),+}}(W^{(1)})}{\log 2} = 2.
$$
Since $h(W^{(2)}) = \log \rho$, where $\rho \approx 1.8668$ satisfies $\rho^4 - 2 \rho^3 + \rho - 1 = 0$, the factor map $\pi: W^{(1)} \to W^{(2)}$ must be infinite-to-one if it exists. The fact $W^{(2)}$ has two fixed points, which can be seen from $T^{(2)}$, asserts that there exists an infinite-to-one factor map $\pi: W^{(1)} \to W^{(2)}$ by Theorem \ref{thm-inf-to-1-SFT}. However, it is difficult to find the explicit form of $\pi$.

Since the unique maximal measure of $W^{(2),+}$ is $\mu^{(2),+} = (p_{W^{(2)}}, P_{W^{(2)}})$ with $p_{W^{(2)}} = (0.1888, 0.0658, 0.2294, 0.3524, 0.1636)$ and
$$
P_{W^{(2)}} = \begin{pmatrix}
                0.5357 & 0 & 0.4643 & 0 & 0 \\
                0 & 0 & 1 & 0 & 0 \\
                0 & 0.2870 & 0 & 0.7130 & 0 \\
                0 & 0 & 0 & 0.5357 & 0.4643 \\
                0.5357 & 0 & 0.4643 & 0 & 0 \\
              \end{pmatrix},
$$
the Hausdorff dimension of $W^{(2)}$ is
$$
\dim W^{(2)} = 2 \frac{h_{\mu^{(2),+}}(W^{(2)})}{\log 5} \approx 0.7758.
$$
Since $W^{(2)}$ is mixing, we have
$$
\dim \mathbf{Y}^{(2)} = 2 \frac{h_{\nu^{(2),+}}(\mathbf{Y}^{(2)})}{\log 2} = 2 \frac{h_{\phi^{(2)}\mu^{(2),+}}(\mathbf{Y}^{(2)})}{\log 2} = 2 \frac{h_{\mu^{(2),+}}(W^{(2)})}{\log 2} \approx 1.8012.
$$
As a conclusion, in the present example, an infinite-to-one factor map is associated with a different Hausdorff dimension.
\end{example}

\begin{figure}
\begin{center}
\includegraphics{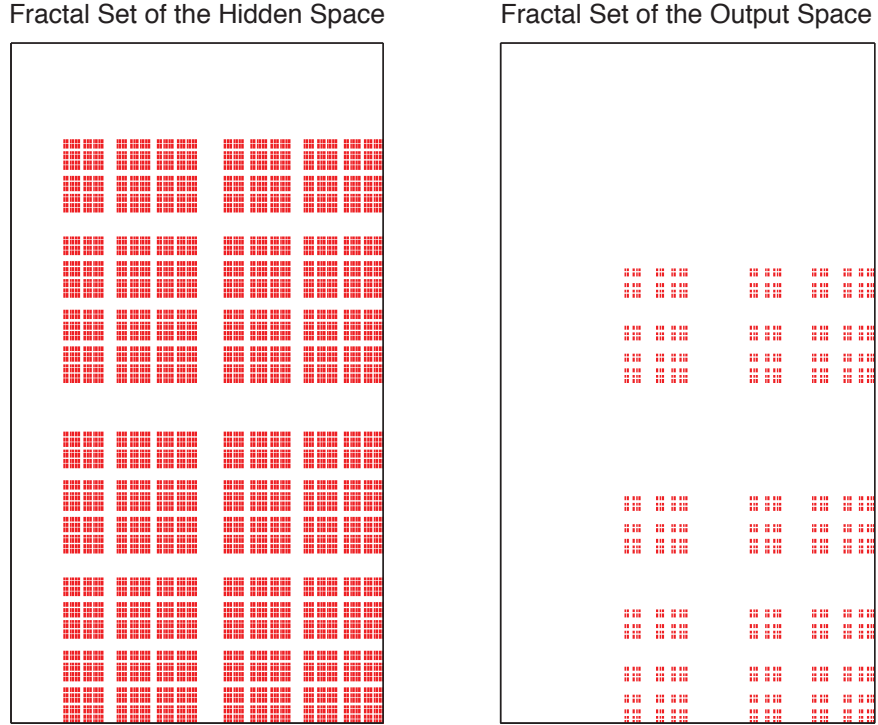}
\caption{The fractal sets of  the hidden and output spaces of Example \ref{eg-ITO-sofic}. The templates are given by $[a^{(1)}, a_r^{(1)}, z^{(1)}] = [2.9, 1.7, 0.1]$ and $[a^{(2)}, a_r^{(2)}, b^{(2)}, b_r^{(2)}, z^{(2)}] = [0.7, -1.1, 2.1, -1.4, 1.7]$. It is demonstrated that there is an infinite-to-one factor map $\pi: W^{(1)} \to W^{(2)}$, and $\mathbf{Y}^{(1)}, \mathbf{Y}^{(2)}$ are strictly sofic.}
\label{fig-642529}
\end{center}
\end{figure}

\begin{example}\label{eg-ITO-sofic}
Suppose the template of the first layer is the same as in Example \ref{eg-Y1Y2-SFT}, and
$$
[a^{(2)}, a_r^{(2)}, b^{(2)}, b_r^{(2)}, z^{(2)}] = [0.7, -1.1, 2.1, -1.4, 1.7].
$$
The basic set of admissible local patterns of the solution space $\mathbf{Y}$ is
$$
\mathcal{B} = \left\{
\boxed{-+ \atop \displaystyle -+}\,, \boxed{-- \atop \displaystyle -+}\,, \boxed{+- \atop \displaystyle +-}\,, \boxed{++ \atop \displaystyle +-}\,, \boxed{+- \atop \displaystyle ++}\,, \boxed{+- \atop \displaystyle --}\,, \boxed{++ \atop \displaystyle ++}\right\}.
$$
The transition matrix $T$ of the solution space $\mathbf{Y}$ is
$$
T = \begin{pmatrix}
        0 & 1 & 0 & 1 \\
        0 & 0 & 0 & 0 \\
        1 & 0 & 0 & 0 \\
        1 & 1 & 1 & 1 \\
      \end{pmatrix}.
$$
A straightforward examination shows that the hidden and output spaces are both mixing with symbolic transition matrices
$$
S^{(1)} = \begin{pmatrix}
            \varnothing & \varnothing & \alpha_1 \\
            \alpha_0 & \varnothing & \alpha_1 \\
            \varnothing & \alpha_2 & \alpha_3 \\
          \end{pmatrix}, \quad
S^{(2)} = \begin{pmatrix}
            \varnothing & \alpha_2 & \alpha_3 \\
            \alpha_1 & \varnothing & \varnothing \\
            \varnothing & \alpha_0 & \alpha_3 \\
          \end{pmatrix}.
$$
$h(\mathbf{Y}^{(1)}) = \log \rho$ and $h(\mathbf{Y}^{(2)}) = \log g$, where $\rho \approx 1.8393$ satisfies $\rho^3 - \rho^2 - \rho - 1 = 0$. See Figure \ref{fig-642529}.

Since $W^{(2)}$ has a fixed point, Theorem \ref{thm-inf-to-1-SFT} infers there is an infinite-to-one factor map $\pi: W{(1)} \to W^{(2)}$. The unique maximal measure of $W^{(1),+}$ is $\mu^{(1),+} = (p_{W^{(1)}}, P_{W^{(1)}})$ with $p_{W^{(1)}} = (0.0994, 0.2822, 0.6184)$ and
$$
P_{W^{(1)}} = \begin{pmatrix}
                0 & 0 & 1 \\
                0.3522 & 0 & 0.6478 \\
                0 & 0.4563 & 0.5437 \\
              \end{pmatrix}.
$$
This suggests
$$
\dim W^{(1)} = 2 \frac{h_{\mu^{(1),+}}(W^{(1)})}{\log 3} \approx 1.1094.
$$
The symbolic transition matrix $S^{(1)}$ asserts that every word of length $2$ in $\mathbf{Y}^{(1)}$ is a synchronizing word, hence $\mathbf{Y}^{(1)}$ is topologically conjugated to $W^{(1)}$ and
$$
\dim \mathbf{Y}^{(1)} = 2 \frac{h_{\nu^{(1),+}}(\mathbf{Y}^{(1)})}{\log 2} = 2 \frac{h_{\phi^{(1)} \mu^{(1),+}}(\mathbf{Y}^{(1)})}{\log 2} \approx 1.7582.
$$

On the other hand, it is verified that the unique maximal measure of $W^{(2),+}$ is $\mu^{(2),+} = (p_{W^{(2)}}, P_{W^{(2)}})$ with $p_{W^{(2)}} = (\dfrac{2-g}{3-g}, \dfrac{2-g}{3-g}, \dfrac{g-1}{3-g})$ and
$$
P_{W^{(2)}} = \begin{pmatrix}
                0 & 2-g & g-1 \\
                1 & 0 & 0 \\
                0 & 2-g & g-1 \\
              \end{pmatrix}.
$$
Since every word of length $2$ in $\mathbf{Y}^{(2)}$ is a synchronizing word, we have
$$
\dim W^{(2)} = 2 \frac{h_{\mu^{(2),+}}(W^{(2)})}{\log 3} = 2 \frac{\log g}{\log 3} \approx 0.8760,
$$
and
$$
\dim \mathbf{Y}^{(2)} = 2 \frac{h_{\nu^{(2),+}}(\mathbf{Y}^{(2)})}{\log 2} = 2 \frac{h_{\phi^{(2)} \mu^{(2),+}}(\mathbf{Y}^{(2)})}{\log 2} \approx 1.3884.
$$
\end{example}

\section{Relation Between the Hausdorff Dimension of Two Hidden Spaces}

Theorems \ref{main-thm-FSE} and \ref{main-thm-ITO} can be extended to two spaces that are induced from a general $n$-layer cellular neural network \eqref{eq-general-system} via analogous discussion as in previous sections. Hence we illustrate the results without providing a detailed argument.  The solution space $\mathbf{Y}$ of \eqref{eq-general-system} is determined by
\begin{align*}
\mathcal{B} &\equiv \mathcal{B}(A^{(1)}, \ldots, A^{(n)}, B^{(1)}, \ldots, B^{(n)}, z^{(1)}, \ldots, z^{(n)}) \\
&= \left\{
\boxed{y_{-d}^{(n)} \cdots y_{-1}^{(n)} y_0^{(n)} y_1^{(n)} \cdots y_d^{(n)} \atop {\displaystyle \vdots \atop {\displaystyle y_{-d}^{(2)} \cdots y_{-1}^{(2)} y_0^{(2)} y_1^{(2)} \cdots y_d^{(2)} \atop {\displaystyle y_{-d}^{(1)} \cdots y_{-1}^{(1)} y_0^{(1)} y_1^{(1)} \cdots y_d^{(1)}}}}}\right\} \subseteq \{-1, 1\}^{\mathbb{Z}_{(2d+1) \times n}}.
\end{align*}
For $1 \leq \ell \leq n$, set
$$
\mathcal{L}^{(\ell)}(y_{-d}^{(n)} \cdots y_d^{(n)} \diamond \cdots \diamond y_{-d}^{(1)} \cdots y_d^{(1)}) = y_{-d}^{(\ell)} \cdots y_d^{(\ell)}.
$$
The hidden space $\mathbf{Y}^{(\ell)}$ is then defined by $\mathcal{L}^{(\ell)}$ as before. (For simplicity, we also call $\mathbf{Y}^{(n)}$ a hidden space instead of the output space.) Similarly, $\mathbf{Y}^{(\ell)}$ is a sofic shift with respect to a right-resolving finite-to-one factor map $\phi^{(\ell)}: W^{(\ell)} \to \mathbf{Y}^{(\ell)}$ and a SFT $W^{(\ell)}$. Furthermore, $W^{(\ell)}$ can be described by the transition matrix $T^{(\ell)}$ while $\mathbf{Y}^{(\ell)}$ can be completely described by the symbolic transition matrix $S^{(\ell)}$.

For $1 \leq i, j \leq n$, without the loss of generality, we assume that $h(\mathbf{Y}^{(i)}) \geq h(\mathbf{Y}^{(j)})$ and $\mathcal{A}(\mathbf{Y}^{(i)}) \geq \mathcal{A}(\mathbf{Y}^{(j)})$.

\begin{proposition}
Suppose $h(\mathbf{Y}^{(i)}) = h(\mathbf{Y}^{(j)})$. If there exists a factor-like matrix $E$ such that $S^{(i)} E = E S^{(j)}$, then there are finite-to-one factor maps $\pi_{ij}: W^{(i)} \to W^{(j)}$ and $\overline{\pi}_{ij}: \mathbf{Y}^{(i)} \to \mathbf{Y}^{(j)}$. For the case where $\mathbf{Y}^{(i)}$ and $\mathbf{Y}^{(j)}$ attain distinct topological entropies, there is an infinite-to-one factor map $\pi_{ij}: W^{(i)} \to W^{(j)}$ if $|\mathcal{A}(W^{(i)})| > |\mathcal{A}(\mathbf{Y})|$ and there exists a factor-like matrix $F$ such that $T^{(i)} F = F T$.
\end{proposition}

The relation of the Hausdorff dimension of $\mathbf{Y}^{(i)}$ and $\mathbf{Y}^{(j)}$, if it exists, is organized as follows.

\begin{theorem}\label{main-thm-general}
Suppose $W^{(i)}$ and $W^{(j)}$ are irreducible SFTs, and there exists a factor map $\pi_{ij}: W^{(i)} \to W^{(j)}$.
\begin{enumerate}[\bf {Case} I.]
\item $\mathbf{Y}^{(i)}, \mathbf{Y}^{(j)}$ share the same topological entropy.
\begin{enumerate}[\bf a)]
  \item There is a one-to-one correspondence between $\mathcal{M}_{\max}(W^{(\ell)})$ and $\mathcal{M}_{\max}(\mathbf{Y}^{(\ell)})$, where $\ell = i, j$.

  \item Let $m_{\ell} = |\mathcal{A}(W^{(\ell)})|, n_{\ell} = |\mathcal{A}(\mathbf{Y}^{(\ell)})|$, and $\mu^{(\ell)}$ be a maximal measure of $W^{(\ell)}$. If $\phi^{(i)}$ has a synchronizing word, then
  $$
  \dim W^{(\ell)} = \frac{h_{\mu^{(\ell)}}(W^{(\ell)})}{\log m_{\ell}} \quad \text{and} \quad \dim \mathbf{Y}^{(\ell)} = \frac{h_{\mu^{(\ell)}}(W^{(\ell)})}{\log n_{\ell}}.
  $$

  \item Suppose $\nu^{(\ell)} = \phi^{(\ell)} \mu^{(\ell)}$. If
  $$
  \dim \mathbf{Y}^{(i)} = \frac{h_{\nu^{(i)}}(\mathbf{Y}^{(i)})}{\log n_i},
  $$
  then
  $$
  \dim \mathbf{Y}^{(j)} = \frac{h_{\overline{\pi}\nu^{(i)}}(\mathbf{Y}^{(j)})}{\log n_j} = \frac{h_{\nu^{(j)}}(\mathbf{Y}^{(j)})}{\log n_j}
  $$
  for some $\overline{\pi}$.
\end{enumerate}

\item $\mathbf{Y}^{(i)}, \mathbf{Y}^{(j)}$ are associated with distinct topological entropies.
\begin{enumerate}[\bf a)]
  \item Suppose $\pi_{ij}: W^{(i)} \to W^{(j)}$ is a uniform factor. If
  $$
  \dim W^{(i)} = \frac{h_{\mu^{(i)}}(W^{(i)})}{\log m_i},
  $$
  then
  $$
  \dim W^{(j)} = \frac{h_{\mu^{(j)}}(W^{(j)})}{\log m_{j}} = \frac{h_{\pi\mu^{(i)}}(W^{(j)})}{\log m_j}.
  $$
  \item If $\phi^{(i)}$ has a synchronizing word, then there exists a factor map $\overline{\pi}: \mathcal{M}_{\max}(\mathbf{Y}^{(i)}) \to \mathcal{M}_{\max}(\mathbf{Y}^{(\overline{i})})$.

  \item If
  $$
  \dim \mathbf{Y}^{(i)} = \dfrac{h_{\nu^{(i)}}(\mathbf{Y}^{(i)})}{\log n_i},
  $$
  then
  $$
  \dim \mathbf{Y}^{(j)} = \dfrac{h_{\overline{\pi} \nu^{(i)}}(\mathbf{Y}^{(j)})}{\log n_j}.
  $$
\end{enumerate}
\end{enumerate}
\end{theorem}

We conclude this section via the flow chart (cf.~Figure \ref{fig-flow-chart}), which explains Theorem \ref{main-thm-general} more clearly.

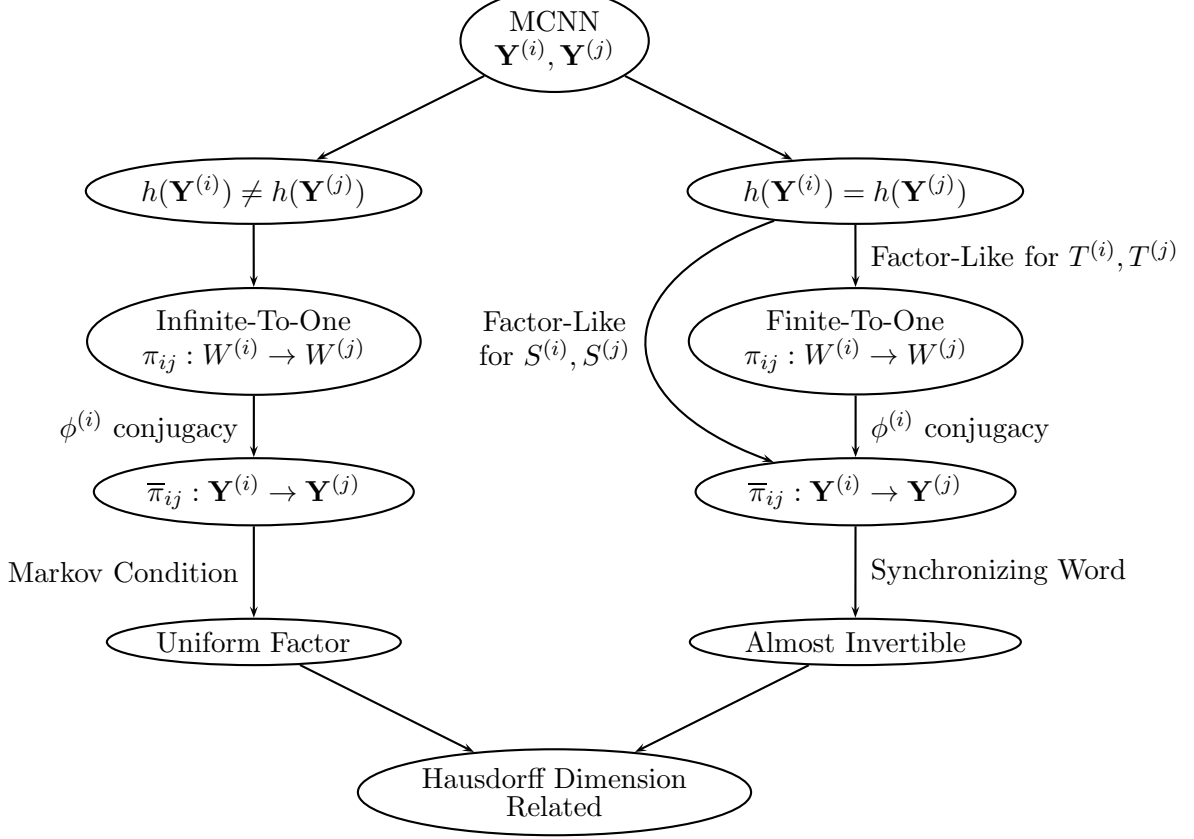
\begin{figure}
\begin{center}
\begin{pspicture}(9,12)
\rput(5,12){\ovalnode{A}{\shortstack{MCNN \\ $\mathbf{Y}^{(i)}, \mathbf{Y}^{(j)}$}}}
\rput(1,10){\ovalnode{B}{$h(\mathbf{Y}^{(i)}) \neq h(\mathbf{Y}^{(j)})$}} \rput(9,10){\ovalnode{C}{$h(\mathbf{Y}^{(i)}) = h(\mathbf{Y}^{(j)})$}}
\rput(1,8){\ovalnode{D}{\shortstack{Infinite-To-One \\ $\pi_{ij}: W^{(i)} \to W^{(j)}$}}} \rput(9,8){\ovalnode{E}{\shortstack{Finite-To-One \\ $\pi_{ij}: W^{(i)} \to W^{(j)}$}}}
\rput(1,6){\ovalnode{F}{$\overline{\pi}_{ij}: \mathbf{Y}^{(i)} \to \mathbf{Y}^{(j)}$}} \rput(9,6){\ovalnode{G}{$\overline{\pi}_{ij}: \mathbf{Y}^{(i)} \to \mathbf{Y}^{(j)}$}}
\rput(1,4){\ovalnode{H}{Uniform Factor}} \rput(9,4){\ovalnode{I}{Almost Invertible}}
\rput(5,2){\ovalnode{J}{\shortstack{Hausdorff Dimension \\ Related}}}

\ncline{->}{A}{B}  \ncline{->}{A}{C}
\ncline{->}{B}{D}  \ncline{->}{C}{E}\Aput{Factor-Like for $T^{(i)}, T^{(j)}$}
\ncline{->}{D}{F}\Bput{$\phi^{(i)}$ conjugacy}  \ncline{->}{E}{G}\Aput{$\phi^{(i)}$ conjugacy}  \ncarc[arcangle=290,ncurv=1.5]{->}{C}{G}\Bput{\shortstack{Factor-Like \\ for $S^{(i)}, S^{(j)}$}}
\ncline{->}{F}{H}\Bput{Markov Condition}  \ncline{->}{G}{I}\Aput{Synchronizing Word}
\ncline{->}{H}{J}  \ncline{->}{I}{J}
%
\end{pspicture}
\end{center}
\caption{The flow chart of the existence of factor maps for arbitrary two hidden spaces.}
\label{fig-flow-chart}
\end{figure}

\section{Conclusion and Further Discussion}

This investigation elucidates whether there is a factor map $\pi$ (respectively $\overline{\pi}$) connecting $W^{(i)}$ and $W^{(j)}$ (respectively $\mathbf{Y}^{(i)}$ and $\mathbf{Y}^{(j)}$). If a factor map does exist, the push-forward measure of a maximal measure is also a maximal measure provided the factor map is either finite-to-one or uniform. Moreover, the Hausdorff dimension of two spaces is thus related. Topological entropy provides a media to make the discussion more clear. 


When $\mathbf{Y}^{(i)}$ and $\mathbf{Y}^{(j)}$ are FSE, the existence of a factor-like matrix asserts the existence of factor map $\overline{\pi}$. With the assistance of computer programs we can rapidly determine if there exists a factor-like matrix for a given MCNN. Moreover, the factor map $\overline{\pi}$ can be expressed in an explicit form. For most of the cases, there is no factor-like matrix for $\mathbf{Y}^{(i)}$ and $\mathbf{Y}^{(j)}$.

\begin{problem}
Suppose there is a factor map between $\mathbf{Y}^{(i)}$ and $\mathbf{Y}^{(j)}$. Is $\dim \mathbf{Y}^{(i)}$ related to $\dim \mathbf{Y}^{(j)}$? Or, equivalently, is there a one-to-one correspondence between $\mathcal{M}_{\max}(\mathbf{Y}^{(i)})$ and $\mathcal{M}_{\max}(\mathbf{Y}^{(j)})$?
\end{problem}

A partial result of the above problem is the existence of synchronizing words. Lemma \ref{lem-ai-sync-word} demonstrates that, if $\phi^{(i)}$/$\phi^{(j)}$ has a synchronizing word, then $\phi^{(i)}$/$\phi^{(j)}$ is almost invertible. This infers a one-to-one correspondence between $\mathcal{M}_{\max}(\mathbf{Y}^{(i)})$ and $\mathcal{M}_{\max}(\mathbf{Y}^{(j)})$.

\begin{problem}
How large is the portion of almost invertible maps in the collection of factor maps?
\end{problem}

If $h(\mathbf{Y}^{(i)}) \neq h(\mathbf{Y}^{(j)})$, on the other hand, we propose a criterion for the existence of factor maps. We will not find the explicit form of the factor map.

\begin{problem}
Can we find some methodology so that we can write down the explicit form of a factor map if it exists?
\end{problem}

For the case where $h(\mathbf{Y}^{(i)}) \neq h(\mathbf{Y}^{(j)})$, a uniform factor provides the one-to-one correspondence between the maximal measures of two spaces. When the Markov condition is satisfied, Theorem \ref{thm-uniform-iff-W} indicates an if-and-only-if criterion. Notably we can use Theorem \ref{thm-uniform-iff-W} only if the explicit form of the factor map is found. Therefore, the most difficult part is the determination of a uniform factor.

\begin{problem}
How to find, in general, a uniform factor?
\end{problem}

\bibliographystyle{amsplain}
\bibliography{../../../grece}

\end{document}